\documentclass[11pt]{amsart}

\usepackage{fullpage}

\usepackage{amsthm, amsmath, amssymb, mathtools}
\swapnumbers
\usepackage[initials, non-sorted-cites]{amsrefs}
\usepackage[hidelinks]{hyperref}
\usepackage[mathscr]{eucal}
\usepackage{enumitem}
\usepackage{verbatim}
\usepackage{tikz}

\setlist[enumerate,1]{label={(\roman*)}}

\allowdisplaybreaks

\numberwithin{equation}{section}

\newtheorem{lemma}[equation]{Lemma}
\newtheorem{proposition}[equation]{Proposition}
\newtheorem{theorem}[equation]{Theorem}
\newtheorem{corollary}[equation]{Corollary}
\newtheorem{remark}[equation]{Remark}

\providecommand{\R}{\mathbb{R}}
\providecommand{\C}{\mathbb{C}}
\providecommand{\Z}{\mathbb{Z}}
\providecommand{\Q}{\mathbb{Q}}

\providecommand{\Sph}{\mathbb{S}}

\providecommand{\st}{\; : \;}

\DeclarePairedDelimiter{\abs}{\lvert}{\rvert}
\DeclarePairedDelimiter{\norm}{\lVert}{\rVert}

\providecommand{\kfam}[1]{\mathcal{S}(#1)}

\providecommand{\ks}{\Xi}

\providecommand{\arcs}[2]{\Theta_{#1}^{[#2]}}
\providecommand{\pts}[1]{\partial\Theta_{#1}}

\DeclareMathOperator{\ind}{ind}
\DeclareMathOperator{\nul}{nul}

\providecommand{\D}{\mathrm{D}}
\providecommand{\N}{\mathrm{N}}
\providecommand{\DN}{\mathrm{Z}}

\DeclareMathOperator{\supp}{supp}
\DeclareMathOperator{\closure}{closure}

\providecommand{\cc}[1]{\overline{#1}}
\renewcommand{\Re}{\operatorname{Re}}
\renewcommand{\Im}{\operatorname{Im}}

\providecommand{\isom}[1]{\mathsf{#1}}
\providecommand{\rot}{\isom{R}}
\providecommand{\refl}{{\underline{\isom{R}}}}
\providecommand{\trans}{\isom{T}}

\providecommand{\ewcover}{\varpi}

\providecommand{\stereop}{\Pi_{\Sph^2}}

\DeclareMathOperator{\dive}{div}
\DeclareMathOperator{\sech}{sech}

\providecommand{\jop}{\mathcal{L}}

\providecommand{\cyl}[1]{{#1}^\mathrm{cyl}}
\providecommand{\sph}[1]{{#1}^\mathrm{sph}}

\providecommand{\sphop}{\jop^{\Sph^2}}
\providecommand{\cylop}{\jop^{\R \times \Sph^1}}

\providecommand{\jext}{\mathcal{E}}
\providecommand{\nop}{\mathcal{N}}

\providecommand{\bop}{\mathcal{B}}

\providecommand{\mult}{\mathcal{M}}
\providecommand{\sgnmul}{\mathcal{S}}
\providecommand{\ident}{\mathcal{I}}
\providecommand{\iop}[1]{\beta_{3,#1}}
\providecommand{\parity}{\mathcal{P}}
\providecommand{\rop}{\mathcal{R}}
\providecommand{\qop}{\mathcal{Q}}

\DeclareMathOperator{\card}{cardinality}

\title{Morse index of Karcher saddle towers
       in $\mathbb{R}^2 \times \mathbb{S}^1(m)$}
\author{David Wiygul}
\address{
  Dipartimento di Matematica,
  Università di Trento,
  via Sommarive 14,
  38123 Povo di Trento,
  Italy
}
\email{davidjames.wiygul@unitn.it}

\begin{document}

\date{\today}

\begin{abstract}
For each integer $k \geq 3$ Hermann Karcher identified
a complete singly periodic minimal surface $\Xi_k$ (unique up to similarity)
with $2k$ ends asymptotic to the union of $k$ planes
intersecting equiangularly along a single line
and with genus zero in the quotient by a fundamental translation.
Writing $\Xi_{k,m}$ for the quotient of $\Xi_k$
by translation through $m \geq 1$ fundamental periods,
we study the Morse index and nullity
of the subfamilies $\Xi_{k,2}$ and $\Xi_{3,m}$.
In the $m=2$ case
we prove
for all $k \geq 3$
that $\Xi_{k,2}$
has Morse index $4k-3$ and nullity $3$.
In the $k=3$ case
we prove that there exists a real number
$\alpha^* \in (1/3,1/2)$
such that
for all $m \geq 1$
the Morse index of $\Xi_{3,m}$
is
$
  6m 
  - 4\lfloor m\alpha^* \rfloor
  - 3 
$
and its nullity is $3$
unless $m\alpha^*$ is an integer,
in which case its nullity is $7$.
Both proofs exploit the symmetries of each surface
and the Dirichlet-to-Neumann map on a half period
to reduce the problem to Fourier-analytic computations
on the unit circle (after a conformal change).
\end{abstract}

\maketitle

\section{Introduction}

In \cite{Scherk}
Scherk famously presented
the first complete embedded minimal surfaces in $\R^3$
with discrete symmetry groups
(unlike the already known plane, catenoid, and helicoid).
One of his examples
is singly periodic
and asymptotic to the union of two orthogonally intersecting planes,
of genus zero in the quotient of $\R^3$
by translation by a fundamental period.
In \cite{Karcher}
Karcher identified this example of Scherk
as just one member of a large family
$\bigcup_{k=2}^\infty \kfam{k}$,
where each $\kfam{k}$ is itself a family,
each of whose members
is a complete embedded
singly periodic minimal surface
asymptotic to the union of $2k$ half planes
and having genus zero in the quotient.
Later in \cite{PerezTraizet}
Pérez and Traizet proved that in fact
every such surface belongs to Karcher's family.
A further reason for interest in these surfaces of Karcher and Scherk
is their role as models in gluing constructions
that desingularize unions of intersecting minimal surfaces
in a variety of settings;
we refer to
\cite{KapouleasCatPlaneDesing},
\cite{TraizetWeierstrass},
and \cite{HellerHellerTraizet}
as pioneering representatives
of the technical approaches in this context.

Here we focus on the maximally symmetric subfamily
whose members' $2k$ asymptotic half planes
lie in $k$ planes, intersecting equiangularly along a line.

\begin{theorem}[\cites{Scherk, Karcher, PerezTraizet}]
\label{ks}
For each integer $k \geq 2$
there exists a unique complete embedded minimal surface
$\ks_k \subset \R^3$ such that
\begin{enumerate}
\item $\ks_k$ is invariant under the translation
      $\trans_z^{2\pi} := (x,y,z) \mapsto (x,y,z+2\pi)$,

\item the quotient $\ks_k / \langle \trans_z^{2\pi} \rangle$
      has genus zero,

\item $\ks_k$ has $2k$ ends, each asymptotic to one of the $k$ planes
      $P^{j'\pi/k} := \{y \cos \tfrac{j'\pi}{k} = x \sin \tfrac{j'\pi}{k}\}$
      with $j' \in \tfrac{1}{2} + \Z$,

\item \label{lines}
      $
        \forall j' \in \tfrac{1}{2} + \Z
        \quad
        P^{j'\pi/k} \cap \ks_k
        =
        P^{j'\pi/k}
          \cap 
          \bigcup_{n' \in \tfrac{1}{2} + \Z} \{z=n'\pi\}
      $,

\item \label{evens}
      $
        \ks_k \cap \{\abs{z} \leq \pi/2 \}
        \subset
        \bigcup_{j' \in \tfrac{1}{2} + 2\Z}
        \{
          (r \cos \theta, r \sin \theta, z)
          \st
          \tfrac{k\theta}{\pi} \in [j',j'+1], ~
          r > 0, ~
          \abs{z} \leq \tfrac{\pi}{2}
        \}
      $.
\end{enumerate}
\end{theorem}

Note that \ref{lines} asserts that
$\ks_k$ intersects the union of its asymptotic planes
in infinitely many straight lines,
the intersection set being precisely
the intersection of the union of the (vertical) asymptotic planes
with the union of infinitely many horizontal planes,
with uniform spacing $\pi$.
The vertical asymptotic planes and these last horizontal planes
partition $\R^3$ into wedges in slabs,
whose interiors $\ks_k$ hits or misses in a checkerboard pattern.
Item \ref{evens} is then a final normalization condition
(after the period, asymptotic planes,
and the heights of their intersection lines with the surface)
that uniquely determines the surface
by declaring which of the two checkerboard colors
covers it.

For any integers $k \geq 2$ and $m \geq 1$ we define
\begin{equation}
  \ks_{k,m}
  :=
  \ks_k / \langle \trans_z^{2m\pi} \rangle,
\end{equation}
the quotient of $\ks_k$ by translation
$\trans_z^{2m\pi} := (\trans_z^{2\pi})^m$
through $m$ fundamental translational periods.
Then $\ks_{k,m}$ is a complete minimal surface
in the quotient
$
  \R^3 / \langle \trans_z^{2m\pi} \rangle
  \cong
  \R^2 \times \Sph^1(m)
$,
of genus $\gamma_{k,m} = (k-1)(m-1)$,
and total Gaussian curvature $-4\pi(k-1)m$.
By \cite{FischerColbrieFiniteIndex}
each $\ks_{k,m}$ has finite Morse index
$\ind(\ks_{k,m})$.
Here we are interested in the question of its precise value
as well as the value of the nullity $\nul(\ks_{k,m})$
of $\ks_{k,m}$,
defined as the dimension (also known to be finite)
of the space of its bounded Jacobi fields.
Equivalently,
we seek the index and nullity of 
$\ks_k$
under $2m\pi$-periodic deformations
or, equivalently still,
the index and nullity of $m^{-1}\ks_k$ in
$
  \R^3 / \langle \trans_z^{2\pi} \rangle
  \cong
  \R^2 \times \Sph^1
$.

The interplay among
the Morse index, topology, and geometry
of minimal submanifolds
has been studied extensively in the literature
(and we refer to \cite{UrbanoBook}
for a systematic presentation of many of the main results to date).
Focusing on our setting,
we highlight the classifications
contained (as special cases of more general results)
in \cite{RitoreIndexOneInFlat}
and \cite{RosOneSided}
of complete minimal surfaces
(assuming two-sidedness and proper embeddedness
in the case of \cite{RitoreIndexOneInFlat})
in $\R^2 \times \Sph^1$
of index one and zero respectively;
\cite{RosOneSided}
also proves lower bounds on the index
which are affine in the genus.
We also mention the exact index and nullity computations
\cites{
  NayataniMorseIndexComplete, NayataniIndexGauss,
  MorabitoCHM,
  KapouleasWlindex,
  KapouleasZouIndexDblI
}
for families of embedded minimal surfaces in other settings;
these surfaces, like ours, are invariant under cyclic groups
whose orders grow with the genus (within a given family),
and the high-genus members of the families
considered in \cite{MorabitoCHM}
and \cite{KapouleasWlindex}
can be naturally regarded as desingularizations,
in the precise sense of
\cite{KapouleasSurveyRS}*{Definition 1.3}.

Most directly relevant to the questions asked here
are certain index and nullity estimates
in \cite{MontielRos}
for complete minimal surfaces of finite total curvature in $\R^3$.
Kapouleas observed
in \cite{KapouleasCatPlaneDesing}
that these results can be applied
to study the kernel of the Jacobi operator
on the $k=2$ models in his construction
(which are not limited
to the maximally symmetric member of $\kfam{2}$).
They can be invoked more generally to solve
our problem's two edge cases
(i) $m=1$ with arbitrary $k \geq 2$
and
(ii) $k = 2$ with arbitrary $m \geq 1$.

\begin{theorem}[by \cite{MontielRos}*{Corollary 15}]
\label{edges}
For any integer $k \geq 2$
\begin{equation*}
  \ind(\ks_{k,1}) = 2k-3 
  \quad \mbox{and} \quad
  \nul(\ks_{k,1}) = 3.
\end{equation*}
For any integer $m \geq 1$
  \begin{equation*}
    \ind(\ks_{2,m}) = 2m-1 = 2\gamma_{2,m} + 1
    \quad \mbox{and} \quad
    \nul(\ks_{2,m}) = 3,
  \end{equation*}
where $\gamma_{2,m} = m-1$ is the genus of $\ks_{2,m}$
(while $\ks_{k,1}$ has genus zero, independently of $k$).
\end{theorem}

\begin{remark}
The case $m=1$ can also be solved more directly.
In fact, $\ks_{k,1}$ has the same Gauss map
as the Jorge--Meeks $k$-noid \cite{JorgeMeeks}
(appropriately pulled back to
$\C \cup \{\infty\}$ under their Enneper--Weierstrass parametrizations);
it follows that Nayatani's computation
in \cite{NayataniLower} of the latter's index and nullity
applies equally to the former.
\end{remark}

\begin{remark}
Actually \ref{edges} holds more generally,
with $\ks_{k,1}$
replaced by $\Xi / \langle \trans_z^{2\pi} \rangle$
for any $\Xi \in \kfam{k}$
and with $\ks_{2,m}$ replaced by
$\Sigma / \langle \trans_z^{2m\pi} \rangle$
for any $\Sigma \in \kfam{2}$.
In the latter case it is clear that
\cite{MontielRos}*{Corollary 15} applies directly.
For the former case,
first,
as observed in \cite{PerezTraizet}*{\S2.1},
it follows from \cite{CosinRos}*{Theorem 4.2}
(and certain symmetry properties of all the surfaces in $\kfam{k}$)
that every member of $\kfam{k}$
has nullity three in the quotient by a fundamental period.
The index equality then follows in turn by \cite{EjiriKotani}*{Theorem A}
or by the connectedness of $\kfam{k}$,
established in \cite{PerezTraizet}.
\end{remark}

\begin{remark}
The lower bound $\nul(\ks_{k,m}) \geq 3$
holds generally,
because the ambient translations
induce Jacobi fields on $\ks_k$
which are bounded
and $2\pi$-periodic.
\end{remark}

The proof of \cite{MontielRos}*{Corollary 15}
is based on a lower bound on the index
and an upper bound on the sum of the index and nullity
which are both sharp under the assumption
that the branch values of the Gauss map
are contained in a single great circle.
This assumption is violated for $k \geq 3$ and $m \geq 2$,
while the lower and upper bounds are still valid
but leave a gap.
In \cite{KapouleasWtordesing}
Kapouleas and I adapted the approach of Montiel and Ros
to study the kernel of the Jacobi operator
on the surfaces $\ks_{k,m}$ (for $k \geq 2$ and $m \geq 1$),
which serve as models in that construction,
there subject also
to a discrete group of rotations about the axis of periodicity.
The same analysis could be applied to obtain not just the nullity
(as needed for our application there) but also the index,
both under the same additional symmetries.
The present article carries out new computations
in the cases $m=2$ and $k=3$
without assuming extra symmetry.

\begin{theorem}[Double-period quotients]
\label{m=2}
For each integer $k \geq 3$
\begin{equation*}
  \ind(\ks_{k,2}) = 4k-3 = 4\gamma_{k,2} + 1
  \quad \mbox{and} \quad
  \nul(\ks_{k,2}) = 3,
\end{equation*}
where $\gamma_{k,2} = k-1$ is the genus of $\ks_{k,2}$.
\end{theorem}

The proof of \ref{m=2} is the subject of Section \ref{Zaremba},
making use of preliminary material presented
in Section \ref{fundamentals}.
The main task is the computation of the index and nullity
on a fundamental half period of $\ks_k$
bounded by two planes of symmetry
and subject to mixed, or Zaremba, boundary data,
Dirichlet on the intersection with one plane
and Neumann on the intersection with the other.
One key step in the solution
is the reduction of this last problem
to the computation of the index and nullity
of the Dirichlet-to-Neumann map
for the Jacobi operator on the half period,
restricted to functions respecting the Zaremba condition.
By taking advantage of the explicit
Enneper--Weierstrass data defining $\ks_k$,
this problem can in turn be transferred to the unit circle
and settled by estimates involving Fourier series
chosen to provide sharp lower and upper bounds
on the index and nullity.

\begin{theorem}[Quotients with six ends]
\label{k=3}
There exists $\alpha^* \in (\tfrac{1}{3}, \tfrac{1}{2})$
such that for every integer $m \geq 1$
\begin{equation*}
  \ind(\ks_{3,m})
  =
  6m  - 4\lfloor m\alpha^* \rfloor - 3
  \quad \mbox{and} \quad
  \nul(\ks_{3,m})
  =
  \begin{cases}
    3 &\mbox{if } m\alpha^* \notin \Z
    \\
    7 &\mbox{if } m\alpha^* \in \Z.
  \end{cases}
\end{equation*}
\end{theorem}

Section \ref{quasiperiodic} takes up the proof of \ref{k=3},
also relying on some definitions and basic results
from Section \ref{fundamentals},
and making important but localized use of some results
of Section \ref{Zaremba}.
The main task is the computation of the index and nullity
on a fundamental period of $\ks_3$
but subject to quasiperiodic boundary conditions
(prescribing a relative phase rather than equality
over a fundamental period);
we do not assume any background
in the theory of periodic elliptic operators,
but we mention \cite{KuchmentOverview}
as a general reference. 
Equivalently
(under a conformal transformation)
we study the index and nullity
of a Schrödinger operator
(a conformal rescaling of the Jacobi operator)
on the unit sphere slit along three equatorial arcs
with a phase shift
corresponding to the quasiperiodic boundary condition
imposed across each arc.
This problem too can be reduced
to the computation of the index and nullity
of the Dirichlet-to-Neumann map
for the Jacobi operator on a fundamental half period,
but now restricted to functions
satisfying conditions
induced by the original quasiperiodic data.
Again, this last problem is formulated
on the unit circle
and resolved by explicit estimates with Fourier series.

The reductions,
for both the $k=3$ and $m=2$ cases,
to the Dirichlet-to-Neumann map
and then to calculations with Fourier series
are the most distinctive aspects of this approach,
in comparison to 
\cites{NayataniMorseIndexComplete, NayataniIndexGauss, MorabitoCHM},
\cite{KapouleasWlindex}, and \cite{KapouleasZouIndexDblI}.
We rely heavily on the explicit Enneper--Weierstrass representation
of $\ks_k$,
in contrast to the computation \cite{KapouleasWlindex}
with Kapouleas (where no such explicit data is available).
Like \cite{KapouleasZouIndexDblI} but unlike \cite{KapouleasWlindex},
we work with functions taking complex
(rather than strictly real) values.

Underlying \ref{k=3} is
the threshold $\alpha^*$:
after the reductions just described,
the dependence on the phase
(corresponding to the quasiperiodic boundary condition)
is carried by a one-parameter family
of operators on the circle,
whose least eigenvalue is continuous and strictly monotonic
in the parameter
and so changes sign just once;
the location of this crossing determines $\alpha^*$.
For each given $m$ only finitely many phases are realized,
and the step function $\lfloor m \alpha^* \rfloor$ in \ref{k=3}
counts those lying below the threshold.
This collapse onto a single threshold
is what the restriction to $k=3$ affords
(just as \ref{m=2} restricts instead to $m=2$
but accommodates every $k$):
for larger $k$
the decomposition by rotational symmetry
produces several inequivalent families of operators,
each in general carrying its own threshold
and requiring its own eigenvalue analysis,
and we do not pursue this here.

\begin{corollary}
For every $m \in \{1,2,3,4,6\}$
\begin{equation*}
  \ind(\ks_{3,m})
  =
  \begin{cases}
    4m-1 = 2\gamma_{3,m} + 3 &\mbox{for } m \in \{1,3\}
    \\
    4m+1 = 2\gamma_{3,m} + 5 &\mbox{for } m \in \{2,4,6\}
  \end{cases}
\quad \mbox{and} \quad
\nul(\ks_{3,m}) = 3,
\end{equation*}
where $\gamma_{3,m} = 2m-2$ is the genus of $\ks_{3,m}$.
\end{corollary}

\begin{proof}
The bounds
$\tfrac{1}{3} < \alpha^* < \tfrac{1}{2}$
imply $m\alpha^* \not\in \Z$
for every $m \in \{1,2,3,4,6\}$
and
$\lfloor \alpha^* \rfloor = 0$,
$\lfloor 2\alpha^* \rfloor = 0$,
$\lfloor 3\alpha^* \rfloor = 1$,
$\lfloor 4\alpha^* \rfloor = 1$,
($\lfloor 5\alpha^* \rfloor \in \{1,2\}$),
and $\lfloor 6\alpha^* \rfloor = 2$.
\end{proof}

\begin{remark}
By \ref{k=3} the index of $\ks_{3,m}$
increases by either six or two as $m$ advances by one
(and the genus correspondingly by two),
depending on whether $\lfloor m\alpha^* \rfloor$
stays constant or jumps;
only if $\alpha^*$ is rational,
which we do not know,
is this increment pattern eventually periodic.

Qualitatively similar behavior---with index
apparently increasing by either four or two
as the genus advances by one---was observed
in the results
\cite{CarlottoSchulzWpairs}*{Table 3}
(briefly discussed in the comments following
\cite{CarlottoSchulzWpairs}*{Conjecture 7.7})
of simulations of some numerical free-boundary minimal surfaces,
all having three boundary components, in the unit ball.
These surfaces belong to a conjectured family
extending one rigorously exhibited
in \cite{CarlottoSchulzWpairs}
by a gluing construction
in the style of \cite{KapouleasCatPlaneDesing},
but using as its desingularizing model
a certain $\Xi \neq \Xi_3 \in \kfam{3}$.

In contrast,
for each of the families treated
in \cites{NayataniMorseIndexComplete, NayataniIndexGauss, MorabitoCHM},
\cite{KapouleasWlindex}, and \cite{KapouleasZouIndexDblI}
the index is shown to be an affine function of the genus.
\end{remark}

\begin{remark}
\label{alpha^*_numerics}
The proof of \ref{k=3} confines $\alpha^*$
to the interval $(\tfrac{1}{3}, \tfrac{1}{2})$,
but in principle the framework
supports its computation to arbitrary precision:
by the monotonicity described above,
any narrowing of this interval
amounts to the determination
of the sign of a least eigenvalue
at certain explicitly identifiable parameter values,
and each such sign
(away from the crossing itself)
is accessible to estimates
of the sort establishing
\ref{threshold_lower_bound} and \ref{threshold_upper_bound}.

Heuristic numerical computation yields
\begin{equation*}
  \alpha^* = 0.4323618(1),
\end{equation*}
the parenthetical digit indicating
the uncertainty in the last place.
For each $m \leq 2298$
no fraction with denominator $m$
falls within the stated uncertainty,
so, with the above estimate,
$\lfloor m \alpha^* \rfloor$ is determined
and $m\alpha^* \notin \Z$,
whence \ref{k=3} identifies
the index and nullity of $\ks_{3,m}$
for every such $m$.

For illustration we include the table below
(nullity being $3$ throughout).
\begin{equation*}
\begin{array}{r|cccccccccccc}
m & 1 & 2 & 3 & 4 & 5 & 6 & 7 & 8 & 9 & 10 & 11 & 12
\\ \hline
\ind(\ks_{3,m}) & 3 & 9 & 11 & 17 & 19 & 25 & 27 & 33 & 39 & 41 & 47 & 49
\\[1ex]
m & 13 & 14 & 15 & 16 & 17 & 18 & 19 & 20 & 21 & 22 & 23 & 24
\\ \hline
\ind(\ks_{3,m}) & 55 & 57 & 63 & 69 & 71 & 77 & 79 & 85 & 87 & 93 & 99 & 101
\end{array}
\end{equation*}
\end{remark}

\subsection*{Acknowledgments}
This project has received funding
from the European Research Council (ERC)
under the European
Union’s Horizon 2020 research and innovation programme
(grant agreement No. 947923).
The questions addressed here
were motivated by prior work
\cite{KapouleasWtordesing}
and
\cite{KapouleasWlindex}
with Nicos Kapouleas;
my interest in the $k=3$ case was heightened
by the project \cite{CarlottoSchulzWpairs}
and in index estimates more generally
by \cites{CarlottoSchulzWgrowth,CarlottoSchulzWspectral,CarlottoSchulzWstackings},
all joint with Alessandro Carlotto and Mario Schulz.

The results were obtained in collaboration
with Claude Fable 5 and Claude Opus 4.8
(from Anthropic).
Each party's involvement was load-bearing,
but the responsibility for any errors is mine alone.
Many elements of the strategy and solution
were in place before I consulted Fable 5 and Opus 4.8,
including the reduction
to the Zaremba problem for \ref{m=2},
to the slit sphere with quasiperiodic data for \ref{k=3},
and to Fourier analysis on the unit circle
via the Dirichlet-to-Neumann map in both cases.
However, I did not have a proof of either \ref{m=2} or \ref{k=3}.
Claude provided not only critical analytic computations
(following many exploratory numerical ones)
but also several key ideas
that I had no hand in generating
(at least in their original forms),
including \ref{single-arc_reduction}---simplifying the proof of
\ref{m=2}---and \ref{iop-qop_equiv}
together with \ref{single_crossing}---essential to the proof of \ref{k=3}.
Virtually all of the writing was done by me,
but there are a few isolated and limited portions of text
which Claude suggested
and I only lightly edited.
Figure \ref{fig:levels} was composed by Claude.

\section{Basic definitions and spectral analysis on a fundamental half period}
\label{fundamentals}

In this section we present the basic geometry
of the $\ks_k$ surfaces,
introduce the Dirichlet-to-Neumann map
for the Jacobi operator on a fundamental half period,
and record for future use the index and nullity there,
subject to either the Dirichlet or Neumann boundary condition,
in each Fourier class
defined under the action of a $\Z_k$ symmetry group.

\subsection*{Conventions}
We take
\begin{equation}
\begin{gathered}
  \Sph^1 := \R / 2\pi\Z,
  \\
  \Sph^2 := \{(x,y,z) \in \R^3 \st x^2+y^2+z^2 = 1\},
  \\
  \Sph^2_+  := \Sph^2 \cap \{z > 0\},
    \quad
    \Sph^2_- := \Sph^2 \cap \{z < 0\}.
\end{gathered}
\end{equation}
When we want to identify
the unit circle in $\C$ with $\Sph^1$
we will explicitly (and pedantically)
do so by applying
\begin{equation}
  \Im \log : \C \setminus \{0\} \to \Sph^1,
\end{equation}
which we will precompose
with stereographic projection
\begin{equation}
\begin{gathered}
  \stereop : \Sph^2 \to \C \cup \{\infty\}
  \\
  (x,y,z) \mapsto \frac{x+iy}{1-z}
\end{gathered}
\end{equation}
from the north pole 
when we instead want to identify the equator
of $\Sph^2$ with $\Sph^1$.

All our function spaces are vector spaces over $\C$,
but we otherwise use standard notation
(e.g. $C^\infty(\Omega) := C^\infty(\Omega;\C)$);
correspondingly the dimension
$\dim X := \dim_{\C} X$
of a subspace $X$ of a function space
always refers to the dimension over $\C$.
All our inner products are conjugate-linear in the left argument.

For a vector space $Y$ and $X \subseteq Y$
we write $X \leq Y$ to mean that
$X$ is a linear subspace of $Y$.

A Riemannian metric $g$ on a manifold induces
Laplacian $\Delta_g := \dive_g \nabla_g$, the divergence of the gradient.
We fix signs so that an eigenfunction $u$
with eigenvalue $\lambda \in \C$
of a Schrödinger operator $\jop = \Delta_g + V$
satisfies $\jop u = \lambda u$.
(By these conventions the eigenvalues of $-\Delta_g$
on a closed manifold are nonnegative.)

The Morse index and nullity of $\ks_{k,m}$ are
\begin{equation}
\begin{gathered}
  \begin{aligned}
  \ind(\ks_{k,m})
  :=
  \sup
  \{
    &\dim X
    \st
    X \leq C_c^\infty(\ks_{k,m})
    \mbox{ and }
  \\
    &\forall u \in X \setminus \{0\}
    \quad
    \norm{\nabla_{g_k} u}_{L^2(g_k)}^2
    <
    \norm{\abs{A_k}_{g_k} u}_{L^2(g_k)}^2
  \},
  \end{aligned}
  \\
  \nul(\ks_{k,m})
  :=
  \dim
  \{
    u \in C^\infty(\ks_{k,m}) \cap L^\infty(\ks_{k,m})
    \st
    \Delta_{g_k} u + \abs{A_k}_{g_k}^2 u = 0
  \},
\end{gathered}
\end{equation}
where $g_k$ and $A_k$ are the first and second fundamental forms of $\ks_k$.
By \cite{FischerColbrieFiniteIndex}
the index is equally the sum of the dimensions
of the $L^2(\ks_{k,m},g_k)$ eigenspaces
of $-(\Delta_{g_k} + \abs{A_k}_{g_k}^2)$
with strictly negative eigenvalue.

Given functions $u: \Sph^1 \to \C$
and $f: \R \to \C$,
we define their Fourier transforms
$\widehat{u} : \Z \to \C$
and $\widehat{f} : \R \to \C$
by
\begin{equation}
  \widehat{u}(\ell)
    :=
    \frac{1}{2\pi} \int_{\Sph^1} u(\theta)\, e^{-i\ell\theta}\, d\theta
  , \quad
  \widehat{f}(\xi)
    :=
    \frac{1}{2\pi} \int_{\R} f(x)\, e^{-i\xi x}\, dx.
\end{equation}

The Fourier transform extends to distributions $u$ on $\Sph^1$,
and for each $s \in \R$ we write $H^s(\Sph^1)$
for the space of those $u$ with
$
  \norm{u}_{H^s(\Sph^1)}^2
  :=
  \sum_{\ell \in \Z}(1+\ell^2)^s\abs{\widehat{u}(\ell)}^2
  <
  \infty
$,
and we define
$H^{-\infty}(\Sph^1) := \bigcup_{s \in \R} H^s(\Sph^1)$.

Given $\Omega \subseteq \Sph^2$,
we write $H^1(\Omega)$
for the space of $u \in L^2(\Omega)$
with weak gradient $\nabla u \in L^2(\Omega)$
(with respect to the round metric on $\Sph^2$).

\subsection*{Symmetries}
Writing $\trans_z^t$ for translation by $(0,0,t)$,
$\rot_z^t$ for rotation about the $z$-axis through angle $t$,
and $\refl_P$ for reflection through plane $P$,
we observe
(from the uniqueness in Theorem \ref{ks}
or inspection of the Enneper--Weierstrass data
\cite{Karcher}*{2.3.1})
\begin{equation}
  \forall \mathsf{S}
    \in
    \{
      \refl_{\{z=0\}}, ~
      \refl_{\{y=0\}}, ~
      \rot_z^{\pi/k}\trans_z^\pi
    \}
  \quad
  \mathsf{S}\ks_k = \ks_k.
\end{equation}
(In fact these three isometries generate the full symmetry group of $\ks_k$.)
Note in particular that
$\rot_z^{2\pi/k}$ is also a symmetry
and every $\{z=n\pi\}$ with $n \in \Z$
is a plane of symmetry.

\subsection*{Enneper--Weierstrass covering map}
The Enneper--Weierstrass data \cite{Karcher}*{2.3.1}
for $\ks_k$ defines a parametrization
of a fundamental translational period
by the Riemann sphere $\C \cup \{\infty\}$,
yielding the covering map
in the following proposition;
the data can be used to confirm its assertions.

For the statement we define 
$\arcs{k}{0}, \arcs{k}{1} \subset \Sph^1$
by
\begin{equation}
\begin{aligned}
  \arcs{k}{n}
  &:=
  \{\theta \in \Sph^1 \st (-1)^n \cos k\theta < 0\}
  \\
  &=
  \bigcup_{j \in \Z}
    \left(
      \tfrac{(4j+1-2n)\pi}{2k}, \, \tfrac{(4j+3-2n)\pi}{2k}
    \right)
    \big/ 2\pi\Z,
\end{aligned}
\end{equation}
each of whose images $e^{i\arcs{k}{n}}$
is a union of $k$ pairwise congruent and disjoint arcs
in the unit circle in $\C$,
with $\{e^{i\arcs{k}{0}},e^{i\arcs{k}{1}}\}$
a partition of
$\{w \in \C \st \abs{w} = 1 \mbox{ and } w^{2k} \neq -1\}$.
We denote their common boundary by
\begin{equation}
\begin{aligned}
  \pts{k} &:= \partial\arcs{k}{0} = \partial\arcs{k}{1}
  =
  \{j'\pi/k \st j' \in \tfrac{1}{2} + \Z\} / 2\pi\Z
  \\
  &=
  \Im \log \{w^{2k} = -1\}.
\end{aligned}
\end{equation}

We also define, for any interval $I \subset \R$, the subset
\begin{equation}
  \ks_k^I
  :=
  \ks_k \cap \{z \in I\};
\end{equation}
in particular
\begin{equation}
  \ks_k^{(0,\pi)}
  =
  \ks_k \cap \{0 < z < \pi\}
\end{equation}
is the interior of a fundamental half period
and $\ks_k^{[0,\pi]}$ is its closure.

\begin{proposition}[\cite{Karcher}]
\label{ewcover}
For each integer $k \geq 3$
there exists a unique smooth conformal covering map
\begin{equation*}
\ewcover_k:
\ks_k
\to
\{w \in \C \cup \{\infty\} \st w^{2k} \neq -1\}
\end{equation*}
such that
\begin{enumerate}
\item
  $
  \ewcover_k \circ \trans_z^{2\pi}
    =
    \ewcover_k,
  \quad
  \ewcover_k \circ \refl_{\{z=0\}}
    =
    \cc{\ewcover_k}^{-1},
  \quad
  \ewcover_k \circ \refl_{\{y=0\}}
    =
    \cc{\ewcover_k}$,
  \\
  $
  \ewcover_k \circ \rot_z^{\pi/k}\trans_z^\pi
    =
    e^{-i\pi/k}\cc{\ewcover_k}^{-1},
  \quad
  \ewcover_k \circ \rot_z^{2\pi/k}
  =
  e^{-2\pi i/ k}\ewcover_k
  $,

\item $\ewcover_k|_{\ks_k^{[0,\pi]}}$
      is a bijection onto
      $\{\abs{w} \leq 1 \st w^{2k} \neq -1\}$,

\item
  $
    \forall n \in \Z
    \quad
    \ewcover_k(\{z=n\pi\})
    =
    e^{i\arcs{k}{n \bmod 2}}
  $,

\item
  $
    \ewcover_k(P^{\pi/(2k)} \cap \{z=\tfrac{\pi}{2}\})
    =
    \{te^{-i\pi/(2k)} \st t \in (-1,1)\}
  $,

\item \label{ewcover-normal}
one choice of continuous unit normal $\nu_k: \ks_k \to \Sph^2$ satisfies
  $
  \stereop \circ \nu_k = \ewcover_k^{k-1}
  $,
  and

\item the first and second fundamental forms
$g_k$ and $A_k$
of $\ks_k$
satisfy
\begin{equation*}   
g_k
  =
  \rho_k^2\ewcover_k^*\abs{dw}^2,
\qquad
\abs{A_k}_{g_k}
  =
  2\sqrt{2} \, \frac{k-1}{\rho_k}
    \, \frac{\abs{\ewcover_k}^{k-2}}{\abs{\ewcover_k}^{2k-2}+1},
\end{equation*}
where the conformal factor $\rho_k: \ks_k \to \R$ is given by
\begin{equation*}
\rho_k
  :=
  k\frac{\abs{\ewcover_k}^{2k-2}+1}{\abs{\ewcover_k^{2k}+1}}.
\end{equation*}
\end{enumerate}
\end{proposition}

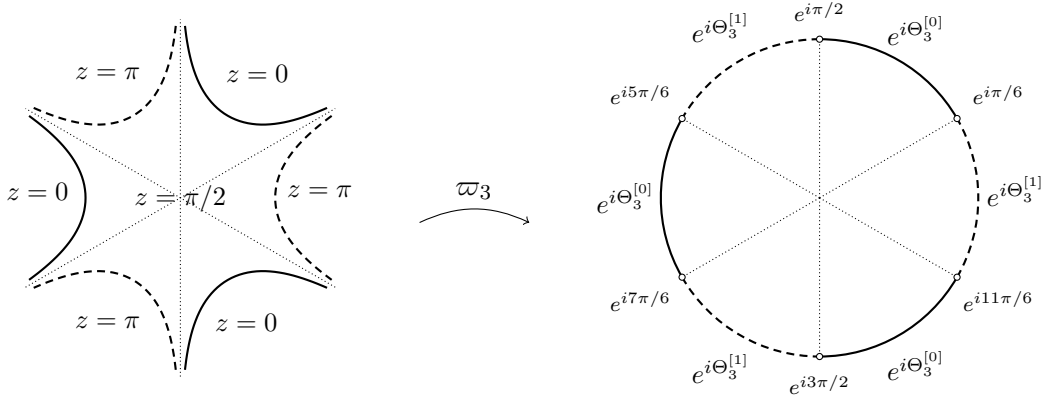
\begin{figure}
\centering
\begin{tikzpicture}[scale=0.55, baseline=(current bounding box.center)]
\draw[thick] plot[smooth] coordinates {(3.528,-2.165) (3.296,-2.067) (3.065,-1.980) (2.836,-1.904) (2.608,-1.842) (2.382,-1.797) (2.159,-1.770) (1.940,-1.763) (1.726,-1.780) (1.521,-1.820) (1.325,-1.885) (1.141,-1.975) (0.970,-2.090) (0.816,-2.227) (0.678,-2.385) (0.557,-2.561) (0.453,-2.754) (0.365,-2.961) (0.291,-3.180) (0.231,-3.408) (0.182,-3.645) (0.142,-3.888) (0.111,-4.138)};
\draw[thick,densely dashed] plot[smooth] coordinates {(3.639,1.973) (3.438,1.821) (3.247,1.665) (3.067,1.504) (2.899,1.337) (2.747,1.164) (2.612,0.985) (2.497,0.798) (2.404,0.605) (2.336,0.407) (2.295,0.204) (2.281,-0.000) (2.295,-0.204) (2.336,-0.407) (2.404,-0.605) (2.497,-0.798) (2.612,-0.985) (2.747,-1.164) (2.899,-1.337) (3.067,-1.504) (3.247,-1.665) (3.438,-1.821) (3.639,-1.973)};
\draw[thick] plot[smooth] coordinates {(0.111,4.138) (0.142,3.888) (0.182,3.645) (0.231,3.408) (0.291,3.180) (0.365,2.961) (0.453,2.754) (0.557,2.561) (0.678,2.385) (0.816,2.227) (0.970,2.090) (1.141,1.975) (1.325,1.885) (1.521,1.820) (1.726,1.780) (1.940,1.763) (2.159,1.770) (2.382,1.797) (2.608,1.842) (2.836,1.904) (3.065,1.980) (3.296,2.067) (3.528,2.165)};
\draw[thick,densely dashed] plot[smooth] coordinates {(-3.528,2.165) (-3.296,2.067) (-3.065,1.980) (-2.836,1.904) (-2.608,1.842) (-2.382,1.797) (-2.159,1.770) (-1.940,1.763) (-1.726,1.780) (-1.521,1.820) (-1.325,1.885) (-1.141,1.975) (-0.970,2.090) (-0.816,2.227) (-0.678,2.385) (-0.557,2.561) (-0.453,2.754) (-0.365,2.961) (-0.291,3.180) (-0.231,3.408) (-0.182,3.645) (-0.142,3.888) (-0.111,4.138)};
\draw[thick] plot[smooth] coordinates {(-3.639,-1.973) (-3.438,-1.821) (-3.247,-1.665) (-3.067,-1.504) (-2.899,-1.337) (-2.747,-1.164) (-2.612,-0.985) (-2.497,-0.798) (-2.404,-0.605) (-2.336,-0.407) (-2.295,-0.204) (-2.281,0.000) (-2.295,0.204) (-2.336,0.407) (-2.404,0.605) (-2.497,0.798) (-2.612,0.985) (-2.747,1.164) (-2.899,1.337) (-3.067,1.504) (-3.247,1.665) (-3.438,1.821) (-3.639,1.973)};
\draw[thick,densely dashed] plot[smooth] coordinates {(-0.111,-4.138) (-0.142,-3.888) (-0.182,-3.645) (-0.231,-3.408) (-0.291,-3.180) (-0.365,-2.961) (-0.453,-2.754) (-0.557,-2.561) (-0.678,-2.385) (-0.816,-2.227) (-0.970,-2.090) (-1.141,-1.975) (-1.325,-1.885) (-1.521,-1.820) (-1.726,-1.780) (-1.940,-1.763) (-2.159,-1.770) (-2.382,-1.797) (-2.608,-1.842) (-2.836,-1.904) (-3.065,-1.980) (-3.296,-2.067) (-3.528,-2.165)};
\draw[thin,densely dotted] (-3.724,-2.150) -- (3.724,2.150);
\draw[thin,densely dotted] (-0.000,-4.300) -- (0.000,4.300);
\draw[thin,densely dotted] (3.724,-2.150) -- (-3.724,2.150);
\node at (1.85,3) {\small $z=0$};
\node at (-3.436,0.152) {\small $z=0$};
\node at (1.586,-3) {\small $z=0$};
\node at (3.35,0.152) {\small $z=\pi$};
\node at (-1.75,3) {\small $z=\pi$};
\node at (-1.75,-3) {\small $z=\pi$};
\node at (0,0) {\small $z=\pi/2$};
\end{tikzpicture}
\hspace{0.5cm}
\begin{tikzpicture}[baseline=(current bounding box.center)]
\draw[->] (0,-0.02) to[out=25,in=155] (1.45,-0.02);
\node at (0.72,0.4) {$\ewcover_3$};
\end{tikzpicture}
\hspace{0.5cm}
\begin{tikzpicture}[scale=2.10, baseline=(current bounding box.center)]
\draw[thick] (30:1) arc (30:90:1);
\draw[thick,densely dashed] (90:1) arc (90:150:1);
\draw[thick] (150:1) arc (150:210:1);
\draw[thick,densely dashed] (210:1) arc (210:270:1);
\draw[thick] (270:1) arc (270:330:1);
\draw[thick,densely dashed] (330:1) arc (330:390:1);
\draw[thin,densely dotted] (0,0) -- (30:1);
\draw[thin,densely dotted] (0,0) -- (90:1);
\draw[thin,densely dotted] (0,0) -- (150:1);
\draw[thin,densely dotted] (0,0) -- (210:1);
\draw[thin,densely dotted] (0,0) -- (270:1);
\draw[thin,densely dotted] (0,0) -- (330:1);
\draw[fill=white] (30:1) circle (0.55pt);
\draw[fill=white] (90:1) circle (0.55pt);
\draw[fill=white] (150:1) circle (0.55pt);
\draw[fill=white] (210:1) circle (0.55pt);
\draw[fill=white] (270:1) circle (0.55pt);
\draw[fill=white] (330:1) circle (0.55pt);
\node at (60:1.22) {\small $e^{i\arcs{3}{0}}$};
\node at (180:1.22) {\small $e^{i\arcs{3}{0}}$};
\node at (300:1.22) {\small $e^{i\arcs{3}{0}}$};
\node at (2:1.24) {\small $e^{i\arcs{3}{1}}$};
\node at (120:1.24) {\small $e^{i\arcs{3}{1}}$};
\node at (240:1.24) {\small $e^{i\arcs{3}{1}}$};
\node[anchor=210] at (30:1.05) {\scriptsize $e^{i\pi/6}$};
\node[anchor=270] at (90:1.05) {\scriptsize $e^{i\pi/2}$};
\node[anchor=330] at (150:1.05) {\scriptsize $e^{i5\pi/6}$};
\node[anchor=30] at (210:1.05) {\scriptsize $e^{i7\pi/6}$};
\node[anchor=90] at (270:1.05) {\scriptsize $e^{i3\pi/2}$};
\node[anchor=150] at (330:1.05) {\scriptsize $e^{i11\pi/6}$};
\end{tikzpicture}
\caption{Level sets of $\ks_3$
and their images under $\ewcover_3$,
at heights $z=0$ (solid),
$z=\pi/2$ (dotted), and $z=\pi$ (dashed).}
\label{fig:levels}
\end{figure}

\subsection*{Jacobi operator}
We denote the Jacobi operator of $\ks_k$
by
\begin{equation}
\begin{aligned}
  \jop_k
  &:=
    \Delta_{g_k} + \abs{A_k}_{g_k}^2
    =
    \dive_{g_k} \nabla_{g_k} + \abs{A_k}_{g_k}^2
  \\
  &=
  \Delta_{g_k}
    +
    \frac{8(k-1)^2}{\rho_k^2}
    \,
    \frac{\abs{\ewcover_k}^{2k-4}}
         {\left(\abs{\ewcover_k}^{2k-2}+1\right)^2}.
\end{aligned}
\end{equation}
Since $\jop_k$ commutes with the (pullback) action of $\trans_z^{2\pi}$,
we innocuously write $\jop_k$ also for the Jacobi operator
of any $\ks_{k,m}$
(and likewise $g_k$ for its metric).

We will need the following basic result.

\begin{lemma}
\label{jacobi_finiteness}
Let $k \geq 3$ and $m \geq 1$ be integers
and
$u \in C^\infty(\ks_{k,m})$ a solution to
$\jop_k u = 0$.
Then $u$ is bounded if and only if
it has finite Dirichlet energy
$\norm{du}_{L^2(g_k)}^2$.
\end{lemma}

\begin{proof}
It suffices to prove the equivalence on a neighborhood of a single end
of $\ks_{k,m}$, say the inverse image $E$ of a small disc centered
at $w_1 := e^{i\frac{\pi}{2k}}$
under the covering induced by $\ewcover_k$.
We choose on $E$ a holomorphic coordinate $\zeta$
such that $\ewcover_k = w_1 + \zeta^m$ on $E$,
under which $E$ is conformally parametrized by a punctured disc,
$
  \norm{du}_{L^2(E, ~ g_k)}
  =
  \norm{du}_{L^2(E, ~ \abs{d\zeta}^2)}
$,
and
the equation $\jop_k u = 0$
is equivalent
(after multiplying both sides by the corresponding conformal factor)
to
\begin{equation}
  \left(
    \partial_\zeta \partial_{\cc{\zeta}}
    +2(k-1)^2 
    \frac{m^2 \abs{\zeta}^{2m-2}\abs{\ewcover_k}^{2k-4}}
         {\left(\abs{\ewcover_k}^{2k-2}+1\right)^2}
    \right)u
  =
  0,
\end{equation}
whose potential is checked to extend smoothly through the puncture
(with value there $0$ for $m \geq 2$ and $(k-1)^2/2$ for $m=1$).
Consequently, the singularity of $u$ at $\zeta=0$ is removable
under the assumption of either property,
and standard interior elliptic regularity then implies the other.
\end{proof}

\subsection*{Spherical picture}
We postcompose with the inverse of stereographic projection
to obtain the covering
\begin{equation}
  \stereop^{-1} \circ \ewcover_k :
    \ks_k \to
    \{
      p \in \Sph^2
      \st
      \stereop(p)^{2k} \neq -1
    \},
\end{equation}
which induces on $\ks_k$ (and $\ks_{k,m}$) the spherical conformal metric
\begin{equation}
  \sph{g}_k := (\stereop^{-1} \circ \ewcover_k)^*g^{\Sph^2}
  =
  \frac{4}{\rho_k^2\left(1+\abs{\ewcover_k}^2\right)^2} \, g_k
\end{equation}
($g^{\Sph^2}$ denoting the standard round metric on $\Sph^2$),
and we define the conformally rescaled Jacobi operator
\begin{equation}
\begin{aligned}
  \sph{\jop}_k
  &:=
  \frac{\rho_k^2\left(1+\abs{\ewcover_k}^2\right)^2}{4} \, \jop_k
  \\
  &=
  \Delta_{\sph{g}_k}
    +
    2(k-1)^2 \,
    \frac{\left(1+\abs{\ewcover_k}^2\right)^2 \abs{\ewcover_k}^{2k-4}}
         {\left(\abs{\ewcover_k}^{2k-2}+1\right)^2},
\end{aligned}
\end{equation}
whose potential term we observe takes values in $[0,2(k-1)^2]$.
Observe also,
from \ref{ewcover}
and the equivariance of stereographic projection,
the intertwining relations
\begin{equation}
\label{sph_symmetries}
\begin{gathered}
  (\stereop^{-1} \circ \ewcover_k) \circ \trans_z^{2\pi}
  =
  \stereop^{-1} \circ \ewcover_k,
  \\
  (\stereop^{-1} \circ \ewcover_k) \circ \refl_{\{z=0\}}
  =
  \refl_{\{z=0\}} \circ (\stereop^{-1} \circ \ewcover_k),
  \\
  (\stereop^{-1} \circ \ewcover_k) \circ \refl_{\{y=0\}}
  =
  \refl_{\{y=0\}} \circ (\stereop^{-1} \circ \ewcover_k),
  \\
  (\stereop^{-1} \circ \ewcover_k) \circ \rot_z^{\pi/k}\trans_z^{\pi}
  =
  \rot_z^{-\pi/k}\refl_{\{z=0\}} \circ (\stereop^{-1} \circ \ewcover_k),
  \\
  (\stereop^{-1} \circ \ewcover_k) \circ \rot_z^{2\pi/k}
  =
  \rot_z^{-2\pi/k} \circ (\stereop^{-1} \circ \ewcover_k),
\end{gathered}
\end{equation}
in which the isometries on the right-hand sides
are understood to act on $\Sph^2 \subset \R^3$.
Each of the isometries of $\Sph^2$ so appearing
preserves $\sphop_k$,
its potential depending only on $\abs{\stereop}$
through a function invariant under
$\abs{\stereop} \mapsto 1/\abs{\stereop}$.

Note that,
according to \ref{ewcover}\ref{ewcover-normal},
this covering of $\Sph^2$ is a
lift of the Gauss map covering of $\Sph^2$
via the branched covering $\Sph^2 \to \Sph^2$
given by $w \mapsto w^{k-1}$.
Note also that
\begin{equation}
  (\stereop^{-1} \circ \ewcover_k)(\ks_k^{[0,\pi]})
  =
  \{
    (x,y,z) \in \Sph^2
    \st
    z \leq 0
  \}
  \setminus
  \{
    (\cos \phi, \sin \phi, 0)
    \st
    \phi \in \pts{k}
  \},
\end{equation}
bijectively.
In particular
$\stereop^{-1} \circ \ewcover_k$
restricts to an isometry
of $(\ks_k^{(0,\pi)}, \sph{g}_k)$
onto the open southern hemisphere.

Last we define on $\Sph^2$ the Schrödinger operator
(with smooth potential)
\begin{equation}
  \sphop_k
  :=
  \Delta_{g^{\Sph^2}}
    +
    V_k^{\Sph^2},
  \qquad
  V_k^{\Sph^2}
  :=
  2(k-1)^2 \,
  \frac{\left(1+\abs{\stereop}^2\right)^2 \abs{\stereop}^{2k-4}}
       {\left(\abs{\stereop}^{2k-2}+1\right)^2},
\end{equation}
so that
\begin{equation}
\label{sphop_intertwine}
  \sph{\jop}_k
    \circ
    (\stereop^{-1} \circ \ewcover_k)^*
  =
  (\stereop^{-1} \circ \ewcover_k)^*
    \circ
    \sphop_k.
\end{equation}

The spherical picture compactifies the half period to a hemisphere,
simplifying the spectral theory.

\subsection*{Cylindrical picture}

We next define the punctured surfaces
\begin{equation}
  \ks_k^\times
    :=
    \ks_k \setminus \{x=y=0\}
    =
    \ks_k \setminus \{(0,0, n'\pi)\}_{n' \in \tfrac{1}{2} + \Z}
\end{equation}
and compose $\ewcover_k$ with the logarithm function
$(\Re \log, ~ \Im \log) : \C \setminus \{0\} \to \R \times \Sph^1$
(recalling $\Sph^1 = \R / 2\pi\Z$)
to obtain the conformal covering
\begin{equation}
  (t_k,\theta_k)
  :=
  (\Re \log \ewcover_k, ~ \Im \log \ewcover_k)
  :
  \ks_k^\times 
  \to
  (\R \times \Sph^1)
    \setminus
    (\{0\} \times \pts{k}).
\end{equation}
We will use $(t_k,\theta_k)$, henceforth written as $(t,\theta)$
with the $k$ understood from context,
as local coordinates ($\Sph^1$-valued in the case of $\theta$)
on $\ks_k^\times$.

Observe that
\begin{equation}
\begin{gathered}
  (t,\theta) \circ \trans_z^{2\pi} = (t,\theta),
  \quad
  (t,\theta) \circ \refl_{\{z=0\}} = (-t,\theta),
  \\
  (t,\theta) \circ \refl_{\{y=0\}} = (t,-\theta),
  \quad
  (t,\theta) \circ \rot_z^{2\pi/k} = (t, \theta - 2\pi/k)
\end{gathered}
\end{equation}
and that
\begin{equation}
  (t,\theta)(\ks_k^{[0,\pi]} \setminus \{(0,0,\tfrac{\pi}{2})\})
  =
  ((-\infty,0] \times \Sph^1) \setminus (\{0\} \times \pts{k}),
\end{equation}
bijectively.

Next we write
\begin{equation}
    \cyl{g}_k
    :=
    \abs{\ewcover_k}^{-2}\rho_k^{-2}g_k
    =
    dt^2 + d\theta^2
\end{equation}
for the cylindrical pullback metric.
Note that
\begin{equation}
  \forall n \in \Z
  \quad
  \cyl{g}_k|_{\{z = n\pi\} \cap \ks_k}
  =
  \sph{g}_k|_{\{z = n\pi\} \cap \ks_k}
\end{equation}
and that
$\partial_t$
is the outward unit conormal
to $\ks_k^{[0,\pi]}$ under the $\cyl{g}_k$ (and $\sph{g}_k$) metric.

Finally we define on $\ks_k^\times$ the conformally rescaled Jacobi operator
\begin{equation}
\begin{aligned}
  \cyl{\jop}_k
  &:=
    \abs{\ewcover_k}^2\rho_k^2 \jop_k
  \\  
  &=
    \Delta_{\cyl{g}_k}
      +8(k-1)^2
        \frac{\abs{\ewcover_k}^{2k-2}}
        {(\abs{\ewcover_k}^{2k-2}+1)^2}
  \\
  &=
    \partial_t^2 + \partial_\theta^2
      + 2(k-1)^2 \sech^2 (k-1)t,
\end{aligned}
\end{equation}
whose final line we take as the definition of
the Schrödinger operator
$\cylop_k$
defined on the entire $\R \times \Sph^1$.
Note that
\begin{equation}
\label{sphop-cylop}
  \sphop_k \circ (\Re \log \stereop, ~ \Im \log \stereop)^*
  =
  \frac{\left(1+\abs{\stereop}^2\right)^2}{4\abs{\stereop}^2}
  \,
  (\Re \log \stereop, ~ \Im \log \stereop)^* \circ \cylop_k,
\end{equation}
the conformal factor taking the constant value $1$ on the equator.

The cylindrical picture
simplifies the form
of the (conformally rescaled) Jacobi operator.

\subsection*{Spectral analysis on the half cylinder}

\begin{proposition}
\label{ode}
Let $k \in \Z \cap [3,\infty)$ and $\mu \in [0,\infty)$.
The space of bounded rotationally invariant solutions to
\begin{equation*}
  \cylop_k u(t) = \mu^2 u(t)
  \quad \mbox{on } (-\infty,0] \times \Sph^1
\end{equation*}
is spanned by
\begin{equation*}
  u_{k,\mu}(t)
  :=
  \left(
    \mu
    -(k-1) \tanh (k-1)t
  \right)
  e^{\mu t}.
\end{equation*}
\end{proposition}

\begin{proof}
The term $2(k-1)^2 \sech^2 (k-1)t$ in $\cylop_k$
is one of the well-known exactly solvable potentials
in the quantum mechanics literature,
whose bound states have been studied
since at least \cite{RosenMorse};
one solution method is briefly presented
at the beginning of the proof of
\cite{KapouleasWtordesing}*{Lemma 6.29}.
Here it suffices to observe
that the full solution space to the ODE is two-dimensional,
and it is readily confirmed that $u_{k,\mu}$
is indeed a bounded solution.
The function obtained by replacing $\mu$ by $-\mu$
in the definition of $u_{k,\mu}$
is also a solution,
which is, however, unbounded (and linearly independent),
except when $\mu \in \{0, k-1\}$,
in which case instead
\begin{equation*}
\begin{gathered}
  u_{k,0'}(t)
    =
    (k-1) t \tanh (k-1)t - 1,
  \\
  u_{k,(k-1)'}(t)
    =
    \sinh (k-1)t + (k-1)t \sech (k-1)t
\end{gathered}
\end{equation*}
provide independent solutions, each clearly unbounded.
(These solutions at $\mu=0$ and $\mu=k-1$
are easily identified geometrically:
in $(t,\theta)$ coordinates on $\ks_k^{[0,\pi]}$ the functions
$u_{k,0}(t)$,
$u_{k,k-1}(t)e^{\pm i(k-1)\theta}$,
$u_{k,0'}(t)$, 
and $u_{k,(k-1)'}(t)e^{\pm i(k-1)\theta}$
are Jacobi fields,
generated respectively by
vertical translation,
horizontal translations,
dilation based at $(0,0,\pi/2)$,
and rotations about lines in $\{z=\pi/2\}$.)
\end{proof}

\subsection*{Half-period Dirichlet-to-Neumann map}

Given $s \in \R$ (in practice $s=0,\pm 1/2$)
we define the mean-zero Sobolev space
\begin{equation}
  \dot{H}^s(\Sph^1)
    :=
    \{
      u \in H^s(\Sph^1)
      \st
      \widehat{u}(0) = 0
    \}.
\end{equation}
Then the operator
\begin{equation}
\label{DtN-momentum}
\begin{gathered}
  \nop_k : \dot{H}^{1/2}(\Sph^1) \to \dot{H}^{-1/2}(\Sph^1)
  \\
  \widehat{\nop_k u}(\ell)
    :=
    \frac{\ell^2-(k-1)^2}{\abs{\ell}}\widehat{u}(\ell)
\end{gathered}
\end{equation}
is well-defined and bounded.

In the statement below (and subsequently too)
we use on $\Sph^2$ less its poles
the cylindrical coordinates
$(t, \theta) := (\Re \log \stereop, \Im \log \stereop)$.
Then
\begin{equation}
\label{sphmet_in_cyl}
  g^{\Sph^2} = (\sech^2 t)(dt^2 + d\theta^2);
\end{equation}
in particular $\partial_t$ is the outward unit conormal
to the hemisphere $\Sph^2_-$ along the equator
$\{t=0\} = \{z=0\}$.
We also introduce the Hermitian form
defined by $\sphop_k$ on $\Sph^2_-$:
\begin{equation}
  \bop_k(u,v)
  :=
  \langle du, dv \rangle_{L^2(\Sph^2_-)}
    - \langle u, V^{\Sph^2}_k v\rangle_{L^2(\Sph^2_-)},
\end{equation}
the Jacobi form on $\ks_k^{(0,\pi)}$ in spherical coordinates.

\begin{proposition}
\label{DtN}
For each integer $k \geq 3$
there is a unique linear map
\begin{equation*}
  \jext_k : \dot{H}^{1/2}(\Sph^1) \to H^1(\Sph^2_-)
\end{equation*}
such that, for every $f \in \dot{H}^{1/2}(\Sph^1)$,
the function $F := \jext_k f$ satisfies
\begin{enumerate}
\item  $\sphop_k F = 0$,

\item  $F(0,\theta) = f(\theta)$ 
       (in the trace sense), and

\item
  the weakly defined Neumann data
  $(\partial_t F)|_{t=0}$
  has vanishing integral on the equator:
  more precisely,
  for every $U \in H^1(\Sph^2_-)$
  whose trace on the equator is constant
  \begin{equation*}
    \bop_k(F,U) = 0;
  \end{equation*}
\end{enumerate}
moreover $\jext_k$ is bounded, and
the weakly defined Neumann data of $F$ satisfies
\begin{enumerate}[resume]
  \item $(\partial_t F)|_{t=0} = (\nop_k f) \circ \theta$:
        more precisely,
        for every $U \in H^1(\Sph^2_-)$,
        writing $u \in H^{1/2}(\Sph^1)$
        for the unique function with
        $U|_{\{t=0\}} = u \circ \theta$
        (in the trace sense),
        we have
        \begin{equation*}
          \bop_k(F,U)
          =
          \langle \nop_k f, \, u \rangle_{L^2(\Sph^1)}
        \end{equation*}
\end{enumerate}
(with the $L^2(\Sph^1)$ inner product extended by continuity
to $H^{-1/2}(\Sph^1) \times H^{1/2}(\Sph^1)$).
\end{proposition}

\begin{proof}
Let $f \in \dot{H}^{1/2}(\Sph^1)$.
Referring to \ref{ode}
and working in the cylindrical $(t,\theta)$ coordinates
on $\Sph^2 \setminus \{z=\pm 1\}$,
we define for each $L \in \Z \cap (0,\infty)$
the function
\begin{equation}
  F_L(t,\theta)
  :=
  \sum_{\{\ell \in \Z \st 0 < \abs{\ell} < L\}}
  \abs{\ell}^{-1}
    \widehat{f}(\ell)
    \, u_{k,\abs{\ell}}(t)
    \, e^{i \ell \theta},
\end{equation}
each summand of which extends smoothly
across the south pole,
so that $F_L$ is smooth
on the closed hemisphere $\{z \leq 0\}$.
By virtue of \ref{sphop-cylop}, \ref{ode},
and the smoothness,
we also have $\sphop_k F_L = 0$.

From \ref{sphmet_in_cyl}
and the pointwise bounds
\begin{equation}
  \abs{u_{k,\abs{\ell}}(t)}
  \leq
  (\abs{\ell} + k) \, e^{\abs{\ell} t},
  \qquad
  \abs{u_{k,\abs{\ell}}'(t)}
  \leq
  (\abs{\ell} + k)^2 \, e^{\abs{\ell} t}
  \qquad
  (t \leq 0)
\end{equation}
we find
\begin{equation}
  \norm{F_L}_{H^1(\Sph^2_-)}^2
  \leq
  C_k
  \sum_{\{\ell \in \Z \st 0 < \abs{\ell} < L\}}
    \abs{\ell} \, \abs{\widehat{f}(\ell)}^2
  \leq
  C_k \norm{f}_{\dot{H}^{1/2}(\Sph^1)}^2
\end{equation}
for a constant $C_k$ depending only on $k$,
the modes being orthogonal
and each mode integral controlled by
$\int_{-\infty}^0 e^{2\abs{\ell}t} \, dt = (2\abs{\ell})^{-1}$.
It follows that
\begin{equation}
  F := \lim_{L \to \infty} F_L \in H^1(\Sph^2_-),
\end{equation}
taking the limit in $H^1$.

Since each $F_L$ satisfies $\sphop_k F_L = 0$ (classically),
the $H^1$ convergence implies (i).
Since $u_{k,\abs{\ell}}(0) = \abs{\ell}$,
each $F_L$ has (classical) trace
\begin{equation}
  F_L(0,\theta)
  =
  f_L(\theta)
  :=
  \sum_{\{\ell \in \Z \st 0 < \abs{\ell} < L\}}
    \widehat{f}(\ell)e^{i \ell \theta},
\end{equation}
and so the $H^1$ convergence (and the trace inequality for $H^1$ functions)
also implies (ii).

Now let $U \in H^1(\Sph^2_-)$
and $u \in H^{1/2}(\Sph^1)$
with $U|_{t=0} = u \circ \theta$.
Since
$u_{k,\abs{\ell}}'(0) = \ell^2 - (k-1)^2$,
each $F_L$ has classical Neumann data
$(\partial_t F_L)|_{t=0} = (\nop_k f_L) \circ \theta$,
so Green's identity for the smooth $F_L$
(extended to $H^1$ test functions by density)
gives
\begin{equation*}
  \bop_k(F_L,U)
  =
  \int_{\Sph^1}
    \cc{(\nop_k f_L)} \, u \, d\theta
  =
  \langle \nop_k f_L, \, u \rangle_{L^2(\Sph^1)}.
\end{equation*}
Letting $L \to \infty$,
the left-hand side converges to that of (iv)
by the $H^1$ convergence $F_L \to F$,
while the right-hand side converges to
$\langle \nop_k f, \, u \rangle_{L^2(\Sph^1)}$,
since $f_L \to f$ in $\dot{H}^{1/2}(\Sph^1)$
and $\nop_k$ is bounded;
this establishes (iv),
and (iii) is the case of constant trace,
the right-hand side of (iv) then vanishing
because $\nop_k f$ has vanishing mean.

For uniqueness suppose $F \in H^1(\Sph^2_-)$
satisfies (i)--(iii) with $f = 0$.
Since (i) holds on the entire open hemisphere
and $\sphop_k$ has smooth coefficients,
elliptic regularity ensures that $F$ is smooth,
in particular bounded near the pole.
Thus for each $\ell \in \Z$
the function $\widehat{F(t,\cdot)}(\ell)$
is a bounded rotationally invariant solution to
$\cylop_k u = \ell^2 u$
on $(-\infty,0) \times \Sph^1$,
which moreover by (ii)
vanishes at $t = 0$.
Using \ref{ode}, we conclude that
$F(t,\theta)$ is a constant multiple
of $\tanh (k-1)t$,
and (iii) then forces $F=0$,
completing the proof.

\end{proof}

\subsection*{Fourier decomposition}
Since $\ks_k$ is invariant under $\rot_z^{2\pi/k}$,
our problem is susceptible to Fourier analysis
over the group $\Z_k$.
We will say that a (complex-valued) function $u$
defined on a $\rot_z^{2\pi/k}$-invariant subset
of $\ks_k$ or $\ks_{k,m}$
belongs to Fourier class $j$
if $u \circ \rot_z^{2\pi/k} = e^{2\pi i j/k}u$.
We also define, for each $s \in \R \cup \{-\infty\}$
and each $j \in \Z_k$
\begin{equation}
\label{Z_k-circle-Sob}
\begin{aligned}
  H^s_{(j;k)}(\Sph^1)
  &:=
  \{
    f \in H^s(\Sph^1)
    \st
    f(\phi - \tfrac{2\pi}{k}) = e^{2\pi i j/k} f(\phi)
  \}
  \\
  &=
  \{
    f \in H^s(\Sph^1)
    \st
    \supp \widehat{f}
    \subseteq
    -j + k\Z
  \}.
\end{aligned}
\end{equation}
Then
\begin{equation}
  H^s(\Sph^1)
    =
    \bigoplus_{j \in \Z_k} H^s_{(j;k)}(\Sph^1),
  \quad
  \dot{H}^s(\Sph^1)
    \supset
    \bigoplus_{j \neq 0} H^s_{(j;k)}(\Sph^1),
\end{equation}
the sums orthogonal in $L^2(\Sph^1)$.

We similarly define
\begin{equation}
  H^1_{(j;k)}(\Sph^2_-)
  :=
  \{
    u \in H^1(\Sph^2_-)
    \st
      u \circ \rot_z^{-2\pi/k}
      =
      e^{2\pi i j/k}u
  \},
\end{equation}
which are also pairwise orthogonal and span $H^1(\Sph^2_-)$.

We further decompose $H^1_{(0;k)}(\Sph^2_-)$
into the closed subspace
\begin{equation}
  H^1_{\Sph^1}(\Sph^2_-)
  :=
  \{
    u \in H^1(\Sph^2_-)
    \st
    \forall \phi \in \R
      \quad
      u \circ \rot_z^\phi = u
  \}
  \leq
  H^1_{(0;k)}(\Sph^2_-)
\end{equation}
of $\Sph^1$-invariant functions
and its $L^2(\Sph^2_-)$ orthogonal complement
(in $H^1_{(0;k)}(\Sph^2_-)$)
\begin{equation}
  H^1_{(0;k),\perp}(\Sph^2_-)
  :=
  \{
    u \in H^1_{(0;k)}(\Sph^2_-)
    \st
    \textstyle{\int u(\cdot,\theta) \, d\theta}
      = 0
  \},
\end{equation}
consisting of functions with zero angular mean.

We observe that the trace map preserves each such Fourier class,
\begin{equation}
  H^1_{(j;k)}(\Sph^2_-) \ni u
  \mapsto
  (u|_{t=0} \circ \theta) \in H^{1/2}_{(j;k)}(\Sph^1)
\end{equation}
as do $\jext_k$ (as can be seen from its uniqueness)
and $\nop_k$,
for $j \neq 0$,
and likewise with
$H^1_{(0;k),\perp}(\Sph^2_-)$ in place of $H^1_{(j;k)}(\Sph^2_-)$
and $\dot{H}^{1/2}(\Sph^1) \cap H^{1/2}_{(0;k)}(\Sph^1)$
in place of $H^{1/2}_{(j;k)}(\Sph^1)$.

\subsection*{Half-period index and nullity}

We define
\begin{equation}
\label{ind_nul_half_period_D}
\begin{gathered}
  \begin{aligned}
  \ind_\D(\ks_k^{(0,\pi)})
  :=
  \sup
  \{
    &\dim X
    \st
    X \leq C_c^\infty(\ks_k^{(0,\pi)})
    \mbox{ and }
  \\
    &\forall u \in X \setminus \{0\}
    \quad
    \norm{\nabla_{g_k} u}_{L^2(g_k)}^2
    <
    \norm{\abs{A_k}_{g_k} u}_{L^2(g_k)}^2
  \},
  \end{aligned}
  \\
  \nul_\D(\ks_k^{(0,\pi)})
  :=
  \dim
  \{
    u \in C^\infty(\ks_k^{[0,\pi]})
    \st
    \jop_k u = 0,
    ~
    \norm{u}_{C^0} < \infty,
    ~
    u|_{\partial \ks_k^{(0,\pi)}} = 0
  \},
\end{gathered}
\end{equation}
and we define $\ind_\N(\ks_k^{(0,\pi)})$
and $\nul_\N(\ks_k^{(0,\pi)})$
by instead taking
$X \leq C_c^\infty(\ks_k^{[0,\pi]})$
and replacing the condition
$u|_{\partial \ks_k^{(0,\pi)}} = 0$
by $(\partial_t u)|_{\partial \ks_k^{(0,\pi)}} = 0$.
Since both boundary conditions
are invariant under $\rot_z^{2\pi/k}$,
we can additionally define,
for each $j \in \Z_k$,
$\ind_{\D,j}(\ks_k^{(0,\pi)})$,
$\nul_{\D,j}(\ks_k^{(0,\pi)})$,
$\ind_{\N,j}(\ks_k^{(0,\pi)})$,
and $\nul_{\N,j}(\ks_k^{(0,\pi)})$
in the same way,
but admitting only variations and Jacobi fields $u$
in the Fourier class $j$,
meaning, we recall,
that
$u \circ \rot_z^{2\pi/k} = e^{2\pi i j/k} u$.

We also define
\begin{equation}
  \nul_{\D,j}(\sphop_k, \Sph^2_-),
  \quad
  \nul_{\N,j}(\sphop_k,\Sph^2_-)
\end{equation}
to be the dimension
of the $j$-equivariant kernel of $\sphop_k$ on the hemisphere
subject to the indicated boundary condition
and
\begin{equation}
  \ind_{\D,j}(\sphop_k, \Sph^2_-),
  \quad
  \ind_{\N,j}(\sphop_k, \Sph^2_-)
\end{equation}
to be the sums of the dimensions
of the $j$-equivariant eigenspaces of $-\sphop_k$
on the hemisphere with strictly negative eigenvalue,
subject to the indicated boundary condition.

\begin{lemma}
\label{junction-density}
For any integers $k \geq 3$ and $0 \leq j \leq k-1$
the functions in $C^\infty(\closure(\Sph^2_-))$
which are of Fourier class $j$
and vanish on a neighborhood of $\{\stereop^{2k} = -1\}$
are dense in $H^1_{(j;k)}(\Sph^2_-)$.
\end{lemma}

\begin{proof}
Truncate the value to obtain a bounded function.
Apply a logarithmic cutoff at each puncture,
symmetrically in $\rot_z^{2\pi/k}$,
and then mollify, also symmetrically.
\end{proof}

\begin{lemma}
\label{puncture-removal}
For integers $k \geq 3$ and $0 \leq j \leq k-1$
let $V$ denote either
$H^1_{(j;k)}(\Sph^2_-)$
or
$H^1_0(\Sph^2_-) \cap H^1_{(j;k)}(\Sph^2_-)$,
and let $u \in V$ satisfy $\bop_k(u,v) = 0$
for every $v \in V$
vanishing on a neighborhood of $\{\stereop^{2k} = -1\}$.
Then $\bop_k(u,v) = 0$ for every $v \in V$.
\end{lemma}

\begin{proof}
Since $V^{\Sph^2}_k$ is bounded and $du \in L^2(\Sph^2_-)$,
the functional $\bop_k(u,\cdot)$ is continuous on $H^1(\Sph^2_-)$,
and it vanishes by hypothesis on those $v \in V$
supported away from $\{\stereop^{2k} = -1\}$,
which are dense in $V$:
for $V = H^1_{(j;k)}(\Sph^2_-)$ by \ref{junction-density},
and for $V = H^1_0(\Sph^2_-) \cap H^1_{(j;k)}(\Sph^2_-)$
because $V$ contains the class-$j$ functions in $C_c^\infty(\Sph^2_-)$.
\end{proof}

\begin{lemma}
\label{index-nullity-conf-comp-equiv}
For any integers $k \geq 3$ and $0 \leq j \leq k-1$
\begin{equation*}
\begin{gathered}
  \nul_{\D,j}(\ks_k^{(0,\pi)})
    =
    \nul_{\D,j}(\sphop_k,\Sph^2_-),
  \quad
  \nul_{\N,j}(\ks_k^{(0,\pi)})
    =
    \nul_{\N,j}(\sphop_k,\Sph^2_-),
  \\
  \ind_{\D,j}(\ks_k^{(0,\pi)})
    =
    \ind_{\D,j}(\sphop_k, \Sph^2_-),
  \quad
  \ind_{\N,j}(\ks_k^{(0,\pi)})
    =
    \ind_{\N,j}(\sphop_k, \Sph^2_-).
\end{gathered}
\end{equation*}
\end{lemma}

\begin{proof}
The map $\stereop^{-1} \circ \ewcover_k$
restricts to a conformal diffeomorphism of $\ks_k^{(0,\pi)}$ onto $\Sph^2_-$,
under which, by the conformal invariance of the Jacobi form,
the latter pulls back to $\bop_k$,
the Dirichlet boundary condition
becomes the vanishing of the trace on the equator,
and the Neumann boundary condition becomes no condition;
accordingly we write $V$ for
$H^1_0(\Sph^2_-) \cap H^1_{(j;k)}(\Sph^2_-)$
in the Dirichlet case
and for $H^1_{(j;k)}(\Sph^2_-)$
in the Neumann case.

For the index,
a class-$j$ subspace of $C_c^\infty(\ks_k^{(0,\pi)})$ (Dirichlet)
or of $C_c^\infty(\ks_k^{[0,\pi]})$ (Neumann)
on which the Jacobi form is negative definite
pushes forward to such a subspace of $V$ for $\bop_k$,
so the min-max characterization of the eigenvalues of $\bop_k$ on $V$
gives
$\ind_{\bullet,j}(\ks_k^{(0,\pi)}) \leq \ind_{\bullet,j}(\sphop_k,\Sph^2_-)$.
Conversely a negative eigenspace of $\bop_k$ on $V$
is approximated in $H^1(\Sph^2_-)$,
preserving both class $j$ and negativity,
by the class-$j$ functions of $C_c^\infty(\Sph^2_-)$ in the Dirichlet case
and by the functions of \ref{junction-density} in the Neumann case;
these pull back to admissible test fields
and yield the reverse inequality.

For the nullity,
let $u$ be a bounded class-$j$ Jacobi field on $\ks_k^{[0,\pi]}$
satisfying the boundary condition.
Odd (Dirichlet) or even (Neumann) reflection across $\{z=0\}$,
permissible since $\refl_{\{z=0\}}$ preserves $\jop_k$,
extends $u$ to a bounded Jacobi field on $\ks_{k,1}$,
of finite Dirichlet energy by \ref{jacobi_finiteness};
hence $u \in V$.
Integrating by parts against $v \in V$
supported away from $\{\stereop^{2k} = -1\}$,
the boundary term vanishes
(the trace of $v$ on the equator in the Dirichlet case,
$\partial_t u$ on it in the Neumann case),
so $\bop_k(u,v) = 0$ for such $v$
and, by \ref{puncture-removal}, for all $v \in V$;
thus $u$ lies in the $j$-equivariant kernel of $\sphop_k$ on $\Sph^2_-$.
Conversely a member of that kernel is,
by elliptic regularity and reflection across the equatorial arcs,
smooth up to the equator away from $\{\stereop^{2k} = -1\}$
and there satisfies the boundary condition,
so it extends to a Jacobi field on $\ks_{k,1}$
of finite Dirichlet energy,
hence bounded by \ref{jacobi_finiteness}.
The two spaces therefore have equal dimension.
\end{proof}

\begin{lemma}
\label{hemi_first_D_eval}
For each integer $k \geq 3$
the least Dirichlet eigenvalue of $-\sphop_k$
on $\Sph^2_-$ is zero and simple,
with eigenspace contained in
$
  H^1_{\Sph^1}(\Sph^2_-)
  \subset
  H^1_{(0;k)}(\Sph^2_-)
$.
\end{lemma}

\begin{proof}
By \ref{ewcover} the $z$-component of the unit normal
(equivalently, in $(t,\theta)$ coordinates,
the function
$u_{k,0}$
of \ref{ode},
up to normalization)
has constant sign on $\ks_k^{(0,\pi)}$
and vanishes on its boundary,
so the (bounded) Jacobi field generated by vertical translations
corresponds under $\stereop^{-1} \circ \ewcover_k$
to a first Dirichlet eigenfunction of $\sphop_k$ on the hemisphere,
invariant under $\Sph^1$ (a fortiori of class $j=0$)
(both manifestly and by virtue of being a first eigenfunction
for an $\Sph^1$-invariant problem).
\end{proof}

For $j \neq 0$
we have
$H^{1/2}_{(j;k)}(\Sph^1) < \dot{H}^{1/2}(\Sph^1)$,
so $(u,v) \mapsto \langle \nop_k u, \, v \rangle_{L^2(\Sph^1)}$
defines a Hermitian form thereon,
whose index and nullity we denote by
\begin{equation*}
\begin{gathered}
  \begin{aligned}
  \ind_{\N,j}(\nop_k)
  :=
  \sup
  \{
    &\dim X
    \st
    X \leq H^{1/2}_{(j;k)}(\Sph^1)
    \mbox{ and }
    \\
    &\;
    \forall u \in X \setminus \{0\}
    \quad
    \langle \nop_k u, \, u \rangle_{L^2(\Sph^1)} < 0
  \},
  \end{aligned}
  \\
  \nul_{\N,j}(\nop_k)
  :=
  \dim
  \{
    u \in H^{1/2}_{(j;k)}(\Sph^1)
    \st
    \forall v \in H^{1/2}_{(j;k)}(\Sph^1)
    \;\,
    \langle \nop_k u, \, v \rangle_{L^2(\Sph^1)} = 0
  \}.
\end{gathered}
\end{equation*}

For $j=0$ we define
\begin{equation}
  \ind_{\N,0,\perp}(\nop_k),
  \quad
  \nul_{\N,0,\perp}(\nop_k)
\end{equation}
by replacing $H^{1/2}_{(j;k)}(\Sph^1)$
by $\dot{H}^{1/2}(\Sph^1) \cap H^{1/2}_{(0;k)}(\Sph^1)$,
and we define
\begin{equation}
\begin{gathered}
  \begin{aligned}
  \ind_{\N,\Sph^1}(\sphop_k, \Sph^2_-)
    :=
    \sum_{\lambda < 0}
      \dim
      \{
        &u \in H^1_{\Sph^1}(\Sph^2_-)
        \st
        \forall v \in H^1_{\Sph^1}(\Sph^2_-)
        \\
        &\;\bop_k(u,v) = \lambda \langle u, v \rangle_{L^2(\Sph^2_-)}
      \},
   \end{aligned}
    \\
    \nul_{\N,\Sph^1}(\sphop_k, \Sph^2_-)
      :=
      \dim
      \{
        u \in H^1_{\Sph^1}(\Sph^2_-)
        \st
        \forall v \in H^1_{\Sph^1}(\Sph^2_-)
        \;
        \bop_k(u,v) = 0
      \},
   \end{gathered}
\end{equation}
as well as the index
$\ind_{\N,0,\perp}(\sphop_k, \Sph^2_-)$
and nullity
$\nul_{\N,0,\perp}(\sphop_k, \Sph^2_-)$
on the orthogonal complement,
by replacing $H^1_{\Sph^1}(\Sph^2_-)$
by $H^1_{(0;k),\perp}(\Sph^2_-)$.
Note that with these definitions
\begin{equation}
  \ind_{\N,0}(\sphop_k, \Sph^2_-)
  =
  \ind_{\N,\Sph^1}(\sphop_k, \Sph^2_-)
    + \ind_{\N,0,\perp}(\sphop_k, \Sph^2_-)
\end{equation}
and likewise with index replaced by nullity.

\begin{lemma}
\label{nop-sphop}
For any integers $k \geq 3$ and $1 \leq j \leq k-1$
\begin{equation*}
  \ind_{\N,j}(\sphop_k, \Sph^2_-)
  =
  \ind_{\N,j}(\nop_k)
  \quad \mbox{and} \quad
  \nul_{\N,j}(\sphop_k, \Sph^2_-)
  =
  \nul_{\N,j}(\nop_k);
\end{equation*}
additionally
\begin{equation*}
  \ind_{\N,0,\perp}(\sphop_k, \Sph^2_-)
    =
    \ind_{\N,0,\perp}(\nop_k)
  \quad \mbox{and} \quad
  \nul_{\N,0,\perp}(\sphop_k, \Sph^2_-)
    =
    \nul_{\N,0,\perp}(\nop_k).
\end{equation*}
\end{lemma}

\begin{proof}
By the min-max characterization of eigenvalues
$\ind_{\N,j}(\sphop_k, \Sph^2_-)$
and $\nul_{\N,j}(\sphop_k, \Sph^2_-)$
are the index and the nullity
of the Hermitian form
$\bop_k$
defined by $\sphop_k$,
restricted to $H^1_{(j;k)}(\Sph^2_-)$.

We define the closed subspaces 
\begin{equation}
  H^1_{J,(j;k)}(\Sph^2_-)
    :=
    \{
      u \in C^\infty(\Sph^2_-) \cap H^1_{(j;k)}(\Sph^2_-)
      \st
      \sphop_k u = 0
    \}
\end{equation}
and observe the direct sum decomposition
\begin{equation}
\label{dirichlet_decomp}
\begin{aligned}
  H^1_{(j;k)}(\Sph^2_-) 
    &=
    H^1_{J,(j;k)}(\Sph^2_-)
      \oplus
      (H^1_0(\Sph^2_-) \cap H^1_{(j;k)}(\Sph^2_-))
  \\
  u(t,\theta)
    &=
    (\jext_k u(0,\cdot))(t,\theta)
      + (u - (\jext_k \, u(0,\cdot)))(t,\theta),
\end{aligned}
\end{equation}
recalling the Jacobi extension operator
$\jext_k$ of \ref{DtN}.
(The sum is direct because
any member of the intersection
is a zero-trace solution in $H^1(\Sph^2_-)$,
so by the uniqueness argument
in the proof of \ref{DtN}
a constant multiple of $\tanh (k-1)t$,
which is $\Sph^1$-invariant
and therefore vanishes if of class $j \neq 0$
(or of zero angular mean).)
Since for any $u,v \in H^1_{(j;k)}(\Sph^2_-)$
\begin{equation}
\label{first_block}
  \bop_k(\jext_k u|_{t=0}, v)
  =
  \langle \nop_k(u(0,\cdot)), v(0,\cdot) \rangle_{L^2(\Sph^1)},
\end{equation}
we see that \ref{dirichlet_decomp}
block diagonalizes $\bop_k$.
Since \ref{hemi_first_D_eval}
implies that the zero-trace block is positive definite,
we conclude that the index and nullity
are those of the $H^1_{J,(j;k)}$ block,
which together with \ref{first_block}
completes the proof
of the $j \neq 0$ equalities.
The remaining claims,
for $j=0$ with zero angular mean,
have the same proof with different notation.
\end{proof}

\begin{proposition}
\label{D_and_N_half_period}
For any integers $k \geq 3$ and $0 \leq j \leq k-1$,
the Dirichlet and Neumann index and nullity of $\ks_k^{(0,\pi)}$
in the class $j$
take the following values.
\begin{equation*}
\begin{array}{r|cccc}
  & j = 0 & j = 1 & 2 \leq j \leq k-2 & j = k-1 \\
\hline
  \ind_{\N,j}(\ks_k^{(0,\pi)}) & 1 & 1 & 2 & 1 \\
  \nul_{\N,j}(\ks_k^{(0,\pi)}) & 0 & 1 & 0 & 1 \\
  \ind_{\D,j}(\ks_k^{(0,\pi)}) & 0 & 0 & 0 & 0 \\
  \nul_{\D,j}(\ks_k^{(0,\pi)}) & 1 & 0 & 0 & 0
\end{array}
\end{equation*}
In particular
\begin{equation*}
\begin{aligned}
  \ind_\N(\ks_k^{(0,\pi)}) &= 2k-3,
  &\qquad
  \nul_\N(\ks_k^{(0,\pi)}) &= 2,
  \\
  \ind_\D(\ks_k^{(0,\pi)}) &= 0,
  &\qquad
  \nul_\D(\ks_k^{(0,\pi)}) &= 1.
\end{aligned}
\end{equation*}
\end{proposition}

\begin{proof}
For all entries we make use of \ref{index-nullity-conf-comp-equiv}
to translate to the corresponding index and nullity
of $\sphop_k$ on $\Sph^2_-$.
The Dirichlet index and nullity for every $j$
then follow from \ref{hemi_first_D_eval},
and the Neumann index and nullity
for every $j \neq 0$
from \ref{nop-sphop}, \ref{DtN-momentum},
and the $\Sph^1$-Fourier characterization
\ref{Z_k-circle-Sob}
of $H^s_{(j;k)}(\Sph^1)$.
In the same way we compute
$
  \ind_{\N,0,\perp}(\sphop_k,\Sph^2_-) 
  =
  \nul_{\N,0,\perp}(\sphop_k,\Sph^2_-)
  =
  0
$,
so it remains only to compute
the $\Sph^1$-invariant Neumann index and nullity.
The Neumann index is at least one,
since the least eigenvalue
is strictly less than the corresponding Dirichlet eigenvalue
which we know to be zero.
Moreover, the least Neumann eigenfunction
is $\Sph^1$-invariant
(also confirmed by the index computations already completed)
since $V^{\Sph^2}_k$ is.
Now let $u$ be a real-valued
second $\Sph^1$-invariant Neumann eigenfunction.
In particular $u$ changes sign.
By the $\Sph^1$ invariance,
it has a nodal curve disjoint from the boundary.
By the strict domain monotonicity of Dirichlet eigenvalues
and the fact that the first Dirichlet eigenvalue is zero
we conclude that the eigenvalue of $u$ is strictly positive.
\end{proof}

\section{Double-period quotients}
\label{Zaremba}

\subsection*{Decomposition}
In this section,
devoted to the proof of \ref{m=2},
we consider,
for any $k \geq 3$, the quotient $\ks_{k,2}$
of $\ks_k$ by two fundamental periods.
The reflectional symmetries through the planes $\{z=n\pi\}$ with $n \in \Z$
descend to the quotient,
on which $\refl_{\{z=n\pi\}}$ and $\refl_{\{z=(n+2)\pi\}}$
induce the same isometry,
having fixed-point set
$\widetilde{P}_n := (\{z \in (n+2\Z)\pi\})/\langle \trans_z^{4\pi} \rangle$,
a union of two disjoint planes, depending on just $n \bmod 2$.
The set $\widetilde{P}_0$ cuts $\ks_{k,2}$
into two components
exchanged by the isometry induced by $\refl_{\{z=0\}}$
and each congruent to $\ks_k^{(0,2\pi)}$.
Each of these components is in turn cut in half
by $\widetilde{P}_1$,
leaving four connected components,
each congruent to
$\ks_k^{(0,\pi)}$.
Because $m=2$,
the symmetries induced
by $\refl_{\{z=0\}}$ and $\refl_{\{z=\pi\}}$
on the quotient commute.
We therefore have the decomposition
\begin{equation}
\begin{aligned}
  \ind(\ks_{k,2})
  &=
  \ind_\D(\ks_k^{(0,2\pi)})
    +\ind_\N(\ks_k^{(0,2\pi)})
  \\
  &=
  \ind_\D(\ks_k^{(0,\pi)})
    +2\ind_\DN(\ks_k^{(0,\pi)})
    +\ind_\N(\ks_k^{(0,\pi)}),
\end{aligned} 
\end{equation}
where the Dirichlet and Neumann index on $\ks_k^{(0,2\pi)}$
are defined as in \ref{ind_nul_half_period_D}
with the obvious modifications
and
where $\ind_\DN$ instead indicates the index
defined by imposing the Zaremba boundary condition
prescribing homogeneous Dirichlet data at $\{z=0\}$ and
homogeneous Neumann data at $\{z=\pi\}$.
The same decomposition applies with the nullity in place of the index.
Applying \ref{D_and_N_half_period} yields
\begin{equation}
\begin{gathered}
  \ind(\ks_{k,2})
    =
    2k-3+2\ind_\DN(\ks_k^{(0,\pi)}),
  \\
  \nul(\ks_{k,2})
    =
    3+2\nul_\DN(\ks_k^{(0,\pi)}).
\end{gathered}
\end{equation}

Since our Zaremba condition is invariant under $\rot_z^{2\pi/k}$,
we can also define, for $j \in \Z_k$,
$\ind_{\DN,j}(\ks_k^{(0,\pi)})$
and $\nul_{\DN,j}(\ks_k^{(0,\pi)})$
in the obvious way.
Working on the spherical compactification
(see \ref{DN-form-equiv}),
each Zaremba eigenvalue is strictly bracketed
by the correspondingly ordered Neumann (below)
and Dirichlet (above) eigenvalues
(using the min-max characterization of eigenvalues
and unique continuation).
This yields
\begin{equation}
\begin{gathered}
  \ind_{\D,j}(\ks_k^{(0,\pi)}) + \nul_{\D,j}(\ks_k^{(0,\pi)})
    \leq
    \ind_{\DN,j}(\ks_k^{(0,\pi)}),
  \\
    \ind_{\DN,j}(\ks_k^{(0,\pi)})
      +\nul_{\DN,j}(\ks_k^{(0,\pi)})
    \leq 
    \ind_{\N,j}(\ks_k^{(0,\pi)}).
\end{gathered}
\end{equation}

In conjunction with \ref{D_and_N_half_period} this implies
\begin{equation}
\label{z-upper}
\begin{gathered}
  \ind_{\DN,0}(\ks_k^{(0,\pi)}) = 1,
  \qquad
  \nul_{\DN,0}(\ks_k^{(0,\pi)}) = 0,
  \\
  \ind_{\DN,j}(\ks_k^{(0,\pi)})
    + \nul_{\DN,j}(\ks_k^{(0,\pi)})
  \leq
  \begin{cases}
    1 &\mbox{if } j \in \{1, k-1\}
    \\
    2 &\mbox{if } 2 \leq j \leq k-2.
  \end{cases}
\end{gathered}
\end{equation}

The proof of \ref{m=2} is then completed
by the next claim,
whose own proof constitutes the remainder of this section.

\begin{proposition}
\label{Zclasses}
For any integers $k \geq 3$ and $1 \leq j \leq k-1$
\begin{equation*}
  \ind_{\DN,j}(\ks_k^{(0,\pi)}) = 1
  \quad \mbox{and} \quad
  \nul_{\DN,j}(\ks_k^{(0,\pi)}) = 0.
\end{equation*}
\end{proposition}

\subsection*{Reduction to the circle}

We define
\begin{equation}
\begin{aligned}
  H^{1/2}_{\DN;k}(\Sph^1)
  &:=
  \{
    f \in H^{1/2}(\Sph^1)
    \st
    \supp f
    \subseteq
    \closure({\arcs{k}{1}})
  \}
  \\
  &=
  \{
    f \in H^{1/2}(\Sph^1)
    \st
    \norm{f|_{\arcs{k}{0}}}_{L^2(\Sph^1)} = 0
  \}
\end{aligned}
\end{equation}
and, for each $j \in \Z \cap [0,k-1]$, its invariant subspaces
\begin{equation}
  H^{1/2}_{\DN, (j;k)}(\Sph^1)
  :=
  H^{1/2}_{\DN;k}(\Sph^1)
  \cap
  H^{1/2}_{(j;k)}(\Sph^1),
\end{equation}
as well as the closed subspace of $H^1(\Sph^2_-)$
\begin{equation}
  H^1_{\DN, (j;k)}(\Sph^2_-)
  :=
  \{
    u \in H^1_{(j;k)}(\Sph^2_-)
    \st
    u(0,\cdot) \in H^{1/2}_{\DN;k}(\Sph^1)
  \}.
\end{equation}
Note that for $1 \leq j \leq k-1$ we have
$H^{1/2}_{\DN, (j;k)}(\Sph^1) \subset \dot{H}^{1/2}(\Sph^1)$.

\begin{lemma}
\label{DN-density}
For integers $k \geq 3$ and $0 \leq j \leq k-1$
the functions in $C^\infty(\closure(\Sph^2_-))$
which are of Fourier class $j$,
vanish on a neighborhood of $\{\stereop^{2k} = -1\}$,
and whose trace $u(0,\cdot)$ vanishes on $\arcs{k}{0}$
are dense in $H^1_{\DN,(j;k)}(\Sph^2_-)$.
\end{lemma}

\begin{proof}
Given $u \in H^1_{\DN,(j;k)}(\Sph^2_-)$,
truncate its value to reduce to $u$ bounded,
the truncation preserving both the class $j$
and the vanishing of the trace on $\arcs{k}{0}$.
Let $\chi_\varepsilon$ be the $\rot_z^{2\pi/k}$-invariant
logarithmic cutoff of \ref{junction-density},
equal to $0$ near $\{\stereop^{2k}=-1\}$ and to $1$ away from it.
Being a bounded real multiplier,
$\chi_\varepsilon$ preserves the vanishing of the trace on $\arcs{k}{0}$,
so $\chi_\varepsilon u \in H^1_{\DN,(j;k)}(\Sph^2_-)$;
and, $u$ being bounded, $\chi_\varepsilon u \to u$ in $H^1(\Sph^2_-)$.
It therefore suffices to approximate each $\chi_\varepsilon u$,
which vanishes near $\{\stereop^{2k}=-1\}$
and has trace supported in a compact subset of $\arcs{k}{1}$,
by functions of the stated form;
since away from the junction points
the Dirichlet and Neumann arcs are disjoint,
this is the standard density of smooth functions
with vanishing trace on $\arcs{k}{0}$,
carried out $\rot_z^{2\pi/k}$-equivariantly.
\end{proof}

\begin{lemma}
\label{DN-puncture-removal}
For integers $k \geq 3$ and $0 \leq j \leq k-1$,
if $u \in H^1_{\DN,(j;k)}(\Sph^2_-)$ satisfies $\bop_k(u,v) = 0$
for every $v \in H^1_{\DN,(j;k)}(\Sph^2_-)$
vanishing on a neighborhood of $\{\stereop^{2k} = -1\}$,
then $\bop_k(u,v) = 0$ for every $v \in H^1_{\DN,(j;k)}(\Sph^2_-)$.
\end{lemma}

\begin{proof}
Since $V^{\Sph^2}_k$ is bounded and $du \in L^2(\Sph^2_-)$,
the functional $\bop_k(u,\cdot)$ is continuous on $H^1(\Sph^2_-)$;
it vanishes by hypothesis on those $v \in H^1_{\DN,(j;k)}(\Sph^2_-)$
supported away from $\{\stereop^{2k}=-1\}$,
which are dense in $H^1_{\DN,(j;k)}(\Sph^2_-)$ by \ref{DN-density}.
\end{proof}

\begin{lemma}
\label{DN-form-equiv}
For any integers $k \geq 3$ and $0 \leq j \leq k-1$
\begin{equation*}
\begin{gathered}
  \begin{aligned}
  \ind_{\DN,j}(\ks_k^{(0,\pi)})
    =
    \sum_{\lambda < 0}
    \dim
    \{
      &u \in H^1_{\DN, (j;k)}(\Sph^2_-)
      \st
      \forall v \in H^1_{\DN, (j;k)}(\Sph^2_-)
      \\
      &\;
      \bop_k(u,v) = \lambda \langle u, v \rangle_{L^2(\Sph^2_-)}
    \},
  \end{aligned}
  \\
  \nul_{\DN,j}(\ks_k^{(0,\pi)})
    =
    \dim
    \{
      u \in H^1_{\DN, (j;k)}(\Sph^2_-)
      \st
      \forall v \in H^1_{\DN, (j;k)}(\Sph^2_-)
      \;\,
      \bop_k(u,v) = 0
    \}.
\end{gathered}
\end{equation*}
\end{lemma}

\begin{proof}
The proof is that of \ref{index-nullity-conf-comp-equiv},
with the single space $H^1_{\DN,(j;k)}(\Sph^2_-)$
(whose $\bop_k$-index and nullity are the two right-hand sides)
in place of the Dirichlet and Neumann spaces,
and \ref{DN-density}, \ref{DN-puncture-removal}
in place of \ref{junction-density}, \ref{puncture-removal}.
Additionally, for the nullity,
since the Zaremba condition is odd across $\{z=0\}$
and even across $\{z=\pi\}$,
we now extend to a Jacobi field on $\ks_{k,2}$
(anti-invariant under $\trans_z^{2\pi}$)
rather than $\ks_{k,1}$,
in order to apply \ref{jacobi_finiteness}.
\end{proof}

Next we define, for $1 \leq j \leq k-1$,
\begin{equation}
\begin{gathered}
  \ind_{\DN,j}(\nop_k)
  :=
  \sup
  \{
    \dim X
    \st
    X \leq H^{1/2}_{\DN,(j;k)}(\Sph^1)
    \mbox{ and }
    \forall u \in X \setminus \{0\}
    \;\,
    \langle \nop_k u, \, u \rangle_{L^2(\Sph^1)} < 0
  \},
  \\
  \nul_{\DN,j}(\nop_k)
  :=
  \dim
  \{
    u \in H^{1/2}_{\DN,(j;k)}(\Sph^1)
    \st
    \forall v \in H^{1/2}_{\DN,(j;k)}(\Sph^1)
    \;\,
    \langle \nop_k u, \, v \rangle_{L^2(\Sph^1)} = 0
  \}.
\end{gathered}
\end{equation}

\begin{proposition}
\label{z-reduction}
For any integers $k \geq 3$ and $1 \leq j \leq k-1$
\begin{equation*}
  \ind_{\DN,j}(\ks_k^{(0,\pi)})
  =
  \ind_{\DN,j}(\nop_k)
  \quad \mbox{and} \quad
  \nul_{\DN,j}(\ks_k^{(0,\pi)})
  =
  \nul_{\DN,j}(\nop_k).
\end{equation*}
\end{proposition}

\begin{proof}
By \ref{DN-form-equiv}
it suffices to compare
the index and nullity of $\bop_k$
on $H^1_{\DN,(j;k)}(\Sph^2_-)$
with those of
$
  \langle \nop_k \cdot, \, \cdot \rangle_{L^2(\Sph^1)}
$
on $H^{1/2}_{\DN,(j;k)}(\Sph^1)$,
and to this end
the proof of \ref{nop-sphop} goes through
with the following substitutions.
The trace space $H^{1/2}_{(j;k)}(\Sph^1)$ there
is here replaced by
$H^{1/2}_{\DN,(j;k)}(\Sph^1) \subset \dot{H}^{1/2}(\Sph^1)$,
and $H^1_{(j;k)}(\Sph^2_-)$ there
is here replaced by
$H^1_{\DN,(j;k)}(\Sph^2_-)$.
Since membership of the trace in
$H^{1/2}_{\DN,(j;k)}(\Sph^1)$
is precisely the constraint defining
$H^1_{\DN,(j;k)}(\Sph^2_-)$,
the extension operator
$\jext_k$ maps $H^{1/2}_{\DN,(j;k)}(\Sph^1)$
into
$H^1_{\DN,(j;k)}(\Sph^2_-)$,
and the decomposition \ref{dirichlet_decomp}
restricts to
$
  H^1_{\DN,(j;k)}(\Sph^2_-)
  =
  \jext_k H^{1/2}_{\DN,(j;k)}(\Sph^1)
    \oplus
    H^1_{0,(j;k)}(\Sph^2_-)
$,
with $H^1_{0,(j;k)}(\Sph^2_-)$
the class-$j$ zero-trace subspace,
under which $\bop_k$ is here too block diagonal
by \ref{first_block}.
Finally, $\bop_k$ is also positive definite
on $H^1_{0,(j;k)}(\Sph^2_-)$,
since by \ref{index-nullity-conf-comp-equiv}
and \ref{D_and_N_half_period}
(here using $1 \leq j \leq k-1$)
every class-$j$ Dirichlet eigenvalue
of $-\sphop_k$ on $\Sph^2_-$
is strictly positive.
The two claimed equalities follow as before.
\end{proof}

\subsection*{Reduction to a single arc}

Note first that the definition
of $H^{1/2}_{\DN;k}(\Sph^1)$
also makes sense verbatim for $k=1$:
\begin{equation}
  H^{1/2}_{\DN;1}(\Sph^1)
  :=
  \{
    f \in H^{1/2}(\Sph^1)
    \st
    \supp f \subseteq \closure(\arcs{1}{1})
  \},
\end{equation}
where $\arcs{1}{1} := \{\theta \st \cos \theta > 0\}$
is a single arc of length $\pi$.
Let $a \in (0,1)$.
We define the operator
\begin{equation}
\label{nop-ak}
\begin{gathered}
  \nop_{(a;k)} : H^{1/2}(\Sph^1) \to H^{-1/2}(\Sph^1)
  \\
  \widehat{\nop_{(a;k)} v}(\ell)
    :=
    \frac{(\ell+a)^2 - \left(1-\tfrac{1}{k}\right)^2}{\abs{\ell+a}}
    \, \widehat{v}(\ell),
\end{gathered}
\end{equation}
which is well-defined and bounded
(since $a \notin \Z$).
We also define
\begin{equation}
  \ind_{\DN}(\nop_{(a;k)}),
  \quad
  \nul_{\DN}(\nop_{(a;k)})
\end{equation}
by replacing,
in the definitions of
$\ind_{\DN,j}(\nop_k)$ and $\nul_{\DN,j}(\nop_k)$,
the space $H^{1/2}_{\DN,(j;k)}(\Sph^1)$
by $H^{1/2}_{\DN;1}(\Sph^1)$
and the operator $\nop_k$ by $\nop_{(a;k)}$.

\begin{lemma}
\label{single-arc_reduction}
For any integers $k \geq 3$ and $1 \leq j \leq k-1$,
with $a := 1 - \tfrac{j}{k} \in (0,1)$,
\begin{equation*}
  \ind_{\DN,j}(\nop_k)
  =
  \ind_{\DN}(\nop_{(a;k)})
  \quad \mbox{and} \quad
  \nul_{\DN,j}(\nop_k)
  =
  \nul_{\DN}(\nop_{(a;k)}).
\end{equation*}
\end{lemma}

\begin{proof}
An element of $H^{1/2}_{\DN,(j;k)}(\Sph^1)$
has Fourier coefficients supported in
$\{k(\ell+a) \st \ell \in \Z\}$
and is determined by its restriction
to the single arc $\{\abs{\theta} < \tfrac{\pi}{2k}\}$
of $\arcs{k}{1}$.
Rescaling this restriction,
we define the map
\begin{equation}
\begin{gathered}
  H^{1/2}_{\DN,(j;k)}(\Sph^1) \longrightarrow H^{1/2}_{\DN;1}(\Sph^1),
  \\
  f(\phi)
  \mapsto
  f_1(\phi) :=
  \begin{cases}
    f(\phi/k) \, e^{-ia\phi} &\mbox{if } \abs{\phi} < \tfrac{\pi}{2}, \\
    0 &\mbox{otherwise}
  \end{cases}
  \\
  \text{(equivalently }
    \widehat{f_1}(\ell) = \widehat{f}(k(\ell+a))
    \text{)},
\end{gathered}
\end{equation}
an isomorphism (bounded linear map with bounded inverse)
between the stated Sobolev spaces
and an $L^2(\Sph^1)$ isometry.
Since for every $\ell \in \Z$
\begin{equation}
  \frac{\left(k(\ell+a)\right)^2 - (k-1)^2}{\abs{k(\ell+a)}}
  =
  k \,
  \frac{(\ell+a)^2 - \left(1-\tfrac{1}{k}\right)^2}{\abs{\ell+a}},
\end{equation}
we have
$
  \langle \nop_k u, \, v \rangle_{L^2(\Sph^1)}
  =
  k \langle \nop_{(a;k)} u_1, \, v_1 \rangle_{L^2(\Sph^1)}
$
(recalling \ref{DtN-momentum}),
implying the claimed equalities.
\end{proof}

\subsection*{Lower bound}

We define 
$h : \R \to \R$,
and test function
$v : \Sph^1 \to \C$
by
\begin{equation}
\label{test-field-def}
\begin{gathered}
  h(x)
  :=
  \begin{cases}
    \sqrt{\cos x} &\mbox{if } \abs{x} < \pi/2
    \\
    0                  &\mbox{otherwise},
  \end{cases}
  \qquad
  v(\theta)
    :=
    h(\underline{\theta})\, e^{-i\underline{\theta}/2},
\end{gathered}
\end{equation}
where $\underline{\theta}$ is the unique representative
of $\theta \in \R / 2\pi\Z$ in $(-\pi,\pi]$.
Recalling that
$\widehat{h} : \R \to \C$ denotes the Fourier transform of $h$
and using the fact that $h(x)=h(-x)$, we observe
\begin{equation}
\label{test-field-shift-parity}
  \widehat{v}(\ell)
  =
  \widehat{h}(\ell + \tfrac{1}{2})
  =
  \widehat{h}(-\ell-\tfrac{1}{2})
  =
  \widehat{v}(-\ell-1).
\end{equation}

\begin{lemma}
\label{test-field}
With $h$ and $v$ as in \ref{test-field-def},
we have
$v \in H^{1/2}_{\DN;1}(\Sph^1)$,
$\abs{\widehat{v}(0)}^2 = \tfrac{1}{8}$,
$\abs{\widehat{v}(1)}^2 = \tfrac{1}{32}$,
and for every $a \in (0,1)$
\begin{equation*}
  \sum_{\ell \in \Z} \abs{\ell+ a} \, \abs{\widehat{v}(\ell)}^2
  =
  \frac{1}{4}
\end{equation*}
(independently of $a$).
\end{lemma}

\begin{proof}
Note that the claim $v \in H^{1/2}_{\DN;1}(\Sph^1)$
follows immediately from the definition of $v$
and the final claimed equality.
We now compute $\widehat{h}$ on $\tfrac{1}{2} + \Z$.
These values can be expressed
in terms of the Gamma function
using a standard integral identity
(Cauchy's beta integral),
but we give the complete calculation here,
for ease of reference.

We set
\begin{equation}
  g(\xi)
  :=
  2\pi\, \widehat{h}(\xi)
  =
  \int_{-\pi/2}^{\pi/2}
    (\cos x)^{1/2}
    \,
    e^{-i \xi x} 
    \,
  dx
  =
  2\int_0^{\pi/2}
    \sqrt{\cos x} \, \cos \xi x \,
  dx,
\end{equation}
where we have used the fact that $h(-x)=h(x)$ for the last equality.
First,
by elementary substitutions
\begin{equation}
  g(\tfrac{1}{2})
  =
  \tfrac{\pi}{\sqrt{2}},
\end{equation}
confirming the value
$
  \abs{\widehat{v}(0)}^2
  =
  \abs{\widehat{h}(\tfrac{1}{2})}^2
  =
  \tfrac{1}{8}
$.

Next, by angle addition, twice,
integrating by parts in between,
\begin{equation}
\begin{aligned}
  \xi g(\xi-1) + \xi g(\xi+1)
  &=
  4\int_0^{\pi/2}
    (\cos x)^{3/2} \, \xi \cos \xi x \, dx
  \\
  &=
  6\int_0^{\pi/2}  \sqrt{\cos x} \, \sin x \sin \xi x \, dx
  \\
  &=
  \frac{3}{2}g(\xi-1) - \frac{3}{2}g(\xi+1),
\end{aligned}
\end{equation}
and therefore for all $\ell \in \Z$
\begin{equation}
  (\ell + 2)g(\ell + \tfrac{3}{2})
  =
  (1 - \ell)g(\ell - \tfrac{1}{2}).
\end{equation}

Consequently,
(i)
$
  g(\tfrac{3}{2})
  =
  \tfrac{1}{2}g(\tfrac{1}{2})
  =
  \tfrac{\pi}{2\sqrt{2}}
$,
confirming the value
$
  \abs{\widehat{v}(1)}^2
  =
  \abs{\widehat{h}(\tfrac{3}{2})}^2
  =
  \tfrac{1}{32}
$,
(ii)
$g(\tfrac{5}{2}) = 0$,
establishing inductively that
for every integer $q \geq 0$
\begin{equation}
  \widehat{v}(2q+2) = 0,
\end{equation}
and
(iii)
\begin{equation}
  \frac{g(2q + \tfrac{3}{2})}{g(\tfrac{3}{2})}
  =
  (-1)^q 
    \prod_{p=1}^q
      \frac{2p-1}{2p+2} 
  =
  \frac{(-1)^q \binom{2q}{q}}{4^q(q+1)}
  =
  (-4)^{-q}C_q,
\end{equation}
where
\begin{equation}
  C_q
  :=
  \frac{1}{q+1}\binom{2q}{q}
  =
  \frac{(2q)!}{q! \, (q+1)!}
\end{equation}
are the Catalan numbers,
so that for every integer $q \geq 0$
\begin{equation}
  \abs{\widehat{v}(2q+1)}^2
  =
  \abs{\widehat{h}(2q+\tfrac{3}{2})}^2
  =
  \frac{C_q^2}{32(4)^{2q}}.
\end{equation}

Therefore
(using also \ref{test-field-shift-parity})
for all $a \in (0,1)$
\begin{equation}
\begin{aligned}
  \sum_{\ell \in \Z} \abs{\ell + a} \,
    \abs{\widehat{v}(\ell)}^2
  &=
  \sum_{\ell=0}^\infty
    [\abs{\ell+a} + \abs{-\ell-1+a}] \,
    \abs{\widehat{v}(\ell)}^2
  \\
  &=
  \sum_{\ell=0}^\infty (2\ell+1)
    \, \abs{\widehat{v}(\ell)}^2
  =
  \frac{1}{8}
    +\frac{1}{32}
      \sum_{q=0}^\infty
        \frac{4q+3}{16^q} C_q^2
  \\
  &=
  \frac{1}{8}
    +\frac{1}{32}\sum_{q=0}^\infty
      \frac{4(q+1)^2 - (2q+1)^2}{16^q}C_q^2
  \\
  &=
  \frac{1}{8}
    +\frac{1}{8}\sum_{q=0}^\infty
    \left[
      16^{-q}\binom{2q}{q}^2
      -
      16^{-(q+1)}\binom{2(q+1)}{(q+1)}^2
    \right]
  \\
  &=
  \frac{1}{8} + \frac{1}{8} = \frac{1}{4},
\end{aligned}
\end{equation}
since $\lim_{q \to \infty} 4^{-q}\binom{2q}{q} = 0$.
\end{proof}

\begin{proposition}
\label{lower-bound}
For any integers $k \geq 3$ and $1 \leq j \leq k-1$,
writing $a := 1 - \tfrac{j}{k}$,
\begin{equation*}
  \ind_{\DN}(\nop_{(a;k)}) \geq 1.
\end{equation*}
\end{proposition}

\begin{proof}
Take $v$ as in \ref{test-field-def}.
Using \ref{nop-ak}, \ref{test-field},
and \ref{test-field-shift-parity},
\begin{equation*}
\begin{aligned}
  \frac{1}{2\pi}
  \langle \nop_{(a;k)} v, \, v \rangle_{L^2(\Sph^1)}
  &=
  \sum_{\ell \in \Z}
    \left(
      \abs{\ell+ a}
      - \frac{\left(1 - \frac{1}{k}\right)^2}{\abs{\ell+a}}
    \right)
    \abs{\widehat v(\ell)}^2
  \\
  &=
    \frac{1}{4}
    -
    \left(
      1
      -\frac{1}{k}
    \right)^2
  \sum_{\ell \in \Z}
    \frac{\abs{\widehat v(\ell)}^2}{\abs{\ell+a}}
  \\
  &<
  \frac{1}{4}
    -
    \left(
      1
      -\frac{1}{k}
    \right)^2
  \sum_{\ell=-1}^0
    \frac{\abs{\widehat v(\ell)}^2}{\abs{\ell+a}}
  \\
  &=
  \frac{1}{4}
    -
    \frac{1}{8}
    \left(
      1
      -\frac{1}{k}
    \right)^2
    \frac{1}{a(1-a)}
  \\
  &=
    \frac{1}{4}
    -\frac{1}{8} \,
      \frac{k-1}{j} \, \frac{k-1}{k-j}.
\end{aligned}
\end{equation*}

For $k=3$ this last quantity is
$\tfrac{1}{4}-\tfrac{1}{8}\tfrac{2}{1}\tfrac{2}{2} = 0$
(for both $j=1$ and $j=2$),
while in general $j(k-j) \leq k^2/4$,
so that
$
  \tfrac{1}{8}\tfrac{k-1}{j}\tfrac{k-1}{k-j}
  \geq
  \tfrac{1}{2}\tfrac{(k-1)^2}{k^2}
  >
  \tfrac{1}{4} 
$
for $k \geq 4$.
\end{proof}

In conjunction with our earlier upper bound \ref{z-upper},
this lower bound \ref{lower-bound}
proves the $j=1$ and $j=k-1$
cases of \ref{Zclasses},
so in particular the entirety of the $k=3$ case (with $m=2$).
To close the remaining cases
we need an improved upper bound.

\subsection*{Upper bound}

We assume $k \geq 4$ and $2 \leq j \leq k-2$.
For each $a \in (0,1)$
we define the symbol $\sigma_{(a;k)} : \Z \to \R$ by
\begin{equation}
  \sigma_{(a;k)}(\ell)
  :=
  \frac{(\ell+a)^2 - (1-\tfrac{1}{k})^2}{\abs{\ell+a}}
\end{equation}
and the quadratic form $Q_{(a;k)} : H^{1/2}(\Sph^1) \to \R$ by
\begin{equation}
\label{upper-Q}
  Q_{(a;k)}(u)
  :=
  2\pi \sum_{\ell \neq 0,-1}
    \sigma_{(a;k)}(\ell) \abs{\widehat u(\ell)}^2.
\end{equation}
Note that
for $a(j,k) := 1 - \tfrac{j}{k}$
we have
$\sigma_{(a(j,k);k)}(0) < 0$,
$\sigma_{(a(j,k);k)}(-1) < 0$,
$\sigma_{(a(j,k);k)}(\ell) > 0$ for all $\ell \in \Z \setminus \{0,-1\}$,
and
\begin{equation}
\begin{aligned}
  \frac{1}{2\pi}\langle u, \, \nop_{(a(j,k);k)} u \rangle_{L^2(\Sph^1)}
  =
  &\frac{1}{2\pi}Q_{(a(j,k);k)}(u)
    -\abs{\sigma_{(a(j,k);k)}(0)} \abs{\widehat u(0)}^2 
    \\
    &- \abs{\sigma_{(a(j,k);k)}(-1)} \abs{\widehat u(-1)}^2
\end{aligned}
\end{equation}
for all $u \in H^{1/2}(\Sph^1)$.

\begin{lemma}
\label{Q-lower}
Let $k \geq 4$ and $2 \leq j \leq k-2$ be integers,
and set $a(j,k) := 1-\tfrac{j}{k}$.
With $Q_{(a(j,k);k)}$ as in \ref{upper-Q},
for all $u \in H^{1/2}_{\DN;1}(\Sph^1)$
with $\widehat{u}(0) \neq \widehat{u}(-1)$
we have
\begin{equation*}
  Q_{(a(j,k);k)}(u)
  >
  2\pi \abs{\widehat{u}(0)-\widehat{u}(-1)}^2
    \frac{\abs{\sigma_{(a(j,k);k)}(0)} \abs{\sigma_{(a(j,k);k)}(-1)}}
      {\abs{\sigma_{(a(j,k);k)}(0)} + \abs{\sigma_{(a(j,k);k)}(-1)}}.
\end{equation*}
\end{lemma}

\begin{proof}
Define
$\Psi \in L^2(\Sph^1)$ by
\begin{equation}
  \Psi(\theta)
  :=
  \begin{cases}
    1-e^{-i\theta}
      &\mbox{on } \arcs{1}{1}
    \\
    \tfrac{2-\pi}{2+\pi}(1-e^{-i\theta})
      &\mbox{on } \Sph^1 \setminus \arcs{1}{1},
  \end{cases}
\end{equation}
and let $u \in H^{1/2}_{\DN;1}(\Sph^1)$
with $\widehat{u}(0) \neq \widehat{u}(-1)$.
We compute
\begin{equation}
  \frac{2+\pi}{2}\abs{\widehat\Psi(\ell)}
  =
  \begin{cases}
    0
      &\mbox{on } \{0,-1\}
    \\
    \abs{\ell+1}^{-1}
      &\mbox{on } 2\Z \setminus \{0\}
    \\
    \abs{\ell}^{-1}
      &\mbox{on } (2\Z + 1) \setminus \{-1\}
  \end{cases}
\end{equation}
and (using $\supp u \subseteq \closure{\arcs{1}{1}}$)
\begin{equation}
  \widehat{u}(0) - \widehat{u}(-1)
  =
  \frac{1}{2\pi}\langle \Psi, u \rangle_{L^2(\Sph^1)}
  =
  \sum_{\ell \in \Z} \cc{\widehat{\Psi}(\ell)} \widehat{u}(\ell)
  =
  \sum_{\ell \neq 0,-1} \cc{\widehat{\Psi}(\ell)} \widehat{u}(\ell).
\end{equation}
Consequently,
\begin{equation}
  \abs{\widehat{u}(0) - \widehat{u}(-1)}^2
  \leq
  \left(
    \sum_{\ell \neq 0,-1}
      \frac{\abs{\widehat{\Psi}(\ell)}^2}
      {\sigma_{(a(j,k);k)}(\ell)}
  \right)
  \left(
    \sum_{\ell \neq 0,-1} \sigma_{(a(j,k);k)}(\ell)
      \abs{\widehat{u}(\ell)}^2
  \right),
\end{equation}
so, in view of the definition \ref{upper-Q} of $Q_{(a;k)}$,
it remains to estimate the first sum.

To proceed, we define for each $a \in (0,1)$
\begin{equation}
  \sigma_{(a;\infty)}(\ell)
  :=
  \frac{(\ell+a)^2 - 1}{\abs{\ell+a}}
  =
  \frac{(\ell-1+a)(\ell+1+a)}{\abs{\ell+a}}
  <
  \sigma_{(a;k)}(\ell)
\end{equation}
and
\begin{equation}
\begin{aligned}
  H(a)
  &:=
  \frac{(2+\pi)^2}{4}
    \sum_{\ell \in \{1,-2\}}
      \frac{\abs{\widehat{\Psi}(\ell)}^2}
      {\sigma_{(a;\infty)}(\ell)}
  \\
  &=
  \frac{a+1}{a(a+2)} - \frac{a-2}{(a-3)(a-1)}
  \\
  &=
  \frac{-3(a^2 - a - 1)}{(a^2-a)(a^2-a-6)}.
\end{aligned}
\end{equation}

Then for all $a \in (0,1)$
\begin{equation}
\begin{aligned}
  \sum_{\ell \neq 0, -1}
      \frac{\abs{\widehat{\Psi}(\ell)}^2}
      {\sigma_{(a;\infty)}(\ell)}
  -\frac{4H(a)}{(2+\pi)^2}
  &=
  \sum_{\ell \neq -2,-1,0,1}
    \frac{\abs{\widehat{\Psi}(\ell)}^2}
    {\sigma_{(a;\infty)}(\ell)}
  \leq
  \sum_{\ell \neq -2,-1,0,1}
    \frac{4\abs{\widehat{\Psi}(\ell)}^2}
    {3\abs{\ell+a}}
  \\
  &\leq
  \sum_{\ell \neq -2,-1,0,1}
    \frac{4\abs{\widehat{\Psi}(\ell)}^2}
    {3(\abs{\ell}-1)}
  \\
  &\leq
  \frac{4}{(2+\pi)^2}
  \sum_{\ell \neq -2,-1,0,1}
    \frac{4}
    {3(\abs{\ell}-1)^3}
  \\
  &
  <
  \frac{16}{3(2+\pi)^2}
    \left(
      1
      +2\int_1^\infty \frac{1}{x^3}
    \right)
  =
  \frac{32}{3(2+\pi)^2},
\end{aligned} 
\end{equation}
and so
\begin{equation}
    \sum_{\ell \neq 0,-1}
      \frac{\abs{\widehat{\Psi}(\ell)}^2}
      {\sigma_{(a(j,k);k)}(\ell)}
  <
  \frac{4H(a)}{(2+\pi)^2} + \frac{32}{3(2+\pi)^2}
  <
  \frac{4H(a)}{25} + \frac{32}{75}.
\end{equation}

On the other hand, for any
$a \in [2/k,(k-2)/k]$
\begin{equation}
\begin{aligned}
  \frac{1}{\abs{\sigma_{(a;k)}(0)}}
    +\frac{1}{\abs{\sigma_{(a;k)}(-1)}}
  &>
  \frac{1}{\abs{\sigma_{(a;\infty)}(0)}}
    +\frac{1}{\abs{\sigma_{(a;\infty)}(-1)}}
  =:
  R(a)
  \\
  &=
  -\frac{a}{(a-1)(a+1)} + \frac{a-1}{a(a-2)}
  \\
  &=
  \frac{a^2-a+1}{(a^2-a)(a^2-a-2)}.
\end{aligned}
\end{equation}
To complete the proof it suffices to show
\begin{equation}
  \frac{4}{25}H(a) + \frac{32}{75} < R(a)
\end{equation}
for all $a \in (0,1)$.

Making the substitution
$s = s(a) := a(1-a) \in (0,\tfrac{1}{4}]$,
we have
\begin{equation}
\begin{gathered}
  H(a) = \frac{3(s + 1)}{s(s+6)},
  \quad
  R(a) = \frac{1-s}{s(s+2)},
  \\
  R(a) - H(a)
  =
  \frac{-4s-14}{s^2+8s+12}.
\end{gathered}
\end{equation}
We then find that for all $a \in (0,1)$
\begin{equation}
  R(a) - H(a) \geq -\frac{7}{6},
  \quad
  R(a) \geq \frac{4}{3},
\end{equation}
whence
\begin{equation}
  75R(a) - 12H(a)
  \geq
  70
  >
  32,
\end{equation}
ending the proof.
\end{proof}

The proof of \ref{m=2} is now finished by the following result.

\begin{proposition}
For any integers
$k \geq 4$ and $2 \leq j \leq k-2$,
writing $a(j,k) := 1 - \tfrac{j}{k}$,
\begin{equation*}
  \ind_{\DN}(\nop_{(a(j,k);k)})
    + \nul_{\DN}(\nop_{(a(j,k);k)})
  \leq
  1.
\end{equation*}
\end{proposition}

\begin{proof}
Suppose to the contrary
there exists 
$X \leq H^{1/2}_{\DN;1}(\Sph^1)$
with
$\dim X = 2$
and
$\langle \nop_{(a(j,k);k)} u, \, u \rangle_{L^2(\Sph^1)} \leq 0$
for every $u \in X$.
Then
$u \mapsto (\widehat{u}(0), \widehat{u}(-1))$
is an isomorphism from $X$ to $\C^2$,
since any nonzero $u \in H^{1/2}(\Sph^1)$
with $\widehat{u}(0)=\widehat{u}(-1)=0$
satisfies
$\langle u, \, \nop_{(a(j,k);k)} u \rangle_{L^2(\Sph^1)} > 0$.
In particular there exists a unique
$w \in X$
with
\begin{equation}
  \widehat{w}(0) = \frac{1}{\sigma_{(a(j,k);k)}(0)},
  \quad
  \widehat{w}(-1) = -\frac{1}{\sigma_{(a(j,k);k)}(-1)}.
\end{equation}

Then by \ref{Q-lower} 
\begin{equation}
  \frac{Q_{(a(j,k);k)}(w)}{2\pi}
  >
  \frac{\abs{\sigma_{(a(j,k);k)}(0)} + \abs{\sigma_{(a(j,k);k)}(-1)}}
    {\abs{\sigma_{(a(j,k);k)}(0)} \abs{\sigma_{(a(j,k);k)}(-1)}},
\end{equation}
so
\begin{equation}
  \langle \nop_{(a(j,k);k)} w, \, w \rangle_{L^2(\Sph^1)}
  =
  \frac{-2\pi}{\abs{\sigma_{(a(j,k);k)}(0)}}
    + \frac{-2\pi}{\abs{\sigma_{(a(j,k);k)}(-1)}}
    + Q_{(a(j,k);k)}(w)
  >
  0,
\end{equation}
contradicting the choice of $X$.
\end{proof}

\section{Quotients with six ends}
\label{quasiperiodic}

\subsection*{Fourier decomposition}
This section is devoted to the proof of \ref{k=3},
so we now fix $k=3$ and consider arbitrary $m \geq 1$.
Since $\ks_3$ is invariant under
the commuting symmetries
$\trans_z^{2\pi}$ and $\rot_z^{2\pi/3}$,
our problem is amenable to Fourier analysis
over the group $\Z \oplus \Z_3$.
For each $\alpha \in \R$
and $j \in \Z_3$
we define the function spaces
\begin{equation}
\begin{gathered}
  \C^{\ks_3}_{\alpha,j}
  :=
  \{
    u : \ks_3 \to \C
    \st
    u \circ \rot_z^{2\pi/3} = e^{2j\pi i / 3}u
    \mbox{ and }
    u \circ \trans_z^{2\pi} = e^{2\alpha \pi i}u
  \},
  \\
  C^\infty_{\alpha,j}(\ks_3)
  :=
  C^\infty(\ks_3) \cap \C^{\ks_3}_{\alpha,j},
  \\
  C^\infty_{0,\alpha,j}(\ks_3)
  :=
  \{
    u \in C^\infty_{\alpha,j}(\ks_3)
    \st
    \exists R > 0
    \;
    \supp u \subset \{x^2+y^2 \leq R^2\}
  \}
\end{gathered}
\end{equation}
and the index and nullity
\begin{equation}
\begin{gathered}
  \begin{aligned}
  \ind_{\alpha,j}(\ks_3)
    :=
      \sup
      \big\{
        &\dim X
        \st
        X \leq C_{0,\alpha,j}^\infty(\ks_3)
        \mbox{ and }
        \forall u \in X \setminus \{0\} \quad
        \\
        &\norm{du}_{L^2\big(\ks_3^{[0,2\pi]}, ~ g_3\big)}^2
        < \langle u, \abs{A_3}^2u \rangle_{L^2\big(\ks_3^{[0,2\pi]}, ~ g_3\big)}
      \big\},
  \end{aligned}
  \\
  \nul_{\alpha,j}(\ks_3)
    :=
    \dim
    \{
      u \in C^\infty_{\alpha,j}(\ks_3) \cap L^\infty(\ks_3)
      \st
      \jop_3 u = 0
    \},
\end{gathered}
\end{equation}
and more coarsely
\begin{equation}
\label{j_decomp_given_alpha}
  \ind_\alpha(\ks_3) := \sum_{j=0}^2 \ind_{\alpha, j}(\ks_3),
  \qquad
  \nul_\alpha(\ks_3) := \sum_{j=0}^2 \nul_{\alpha, j}(\ks_3).
\end{equation}

For each integer $m \geq 1$ we then have
\begin{equation}
\label{k=3_decomp}
  \ind(\ks_{3,m})
  =
  \sum_{n=0}^{m-1} \ind_{\tfrac{n}{m}}(\ks_3),
  \qquad
  \nul(\ks_{3,m})
  =
  \sum_{n=0}^{m-1} \nul_{\tfrac{n}{m}}(\ks_3),
\end{equation}
while by \ref{edges}
\begin{equation}
\label{alpha=0}
  \ind_0(\ks_3) = \ind(\ks_{3,1}) = 3,
  \qquad
  \nul_0(\ks_3) = \nul(\ks_{3,1}) = 3.
\end{equation}
Note that for each $\alpha \in \R$
\begin{equation}
\label{j_sym}
  \ind_{\alpha,1}(\ks_3) = \ind_{\alpha,2}(\ks_3),
  \quad
  \nul_{\alpha,1}(\ks_3) = \nul_{\alpha,2}(\ks_3),
\end{equation}
as $\refl_{\{y=0\}}$ is a symmetry of $\ks_3$
taking $\C^{\ks_3}_{\alpha,1}$ to $\C^{\ks_3}_{\alpha,2}$,
and for each $j \in \Z_3$
\begin{equation}
  \ind_{\alpha,j}(\ks_3) = \ind_{1-\alpha,-j}(\ks_3),
  \quad
  \nul_{\alpha,j}(\ks_3) = \nul_{1-\alpha,-j}(\ks_3),
\end{equation}
as complex conjugation takes $\C^{\ks_3}_{\alpha,j}$
to $\C^{\ks_3}_{1-\alpha,-j}$
(and preserves the Jacobi form);
in particular
$\ind_\alpha(\ks_3) = \ind_{1-\alpha}(\ks_3)$
and
$\nul_\alpha(\ks_3) = \nul_{1-\alpha}(\ks_3)$.

\subsection*{Slit sphere}
We define
\begin{equation}
  \Sph^2_3
  :=
  \Sph^2 \setminus \stereop^{-1}(e^{i\closure(\arcs{3}{0})}),
\end{equation}
the sphere slit along three equatorial arcs
$\{(0,\theta) \st \theta \in \closure(\arcs{3}{0})\}$,
recalling the cylindrical coordinates
$(t,\theta) := (\Re \log \stereop, ~ \Im \log \stereop)$
on $\Sph^2 \setminus \{\abs{z}=1\}$.
For each $u \in H^1(\Sph^2_3)$
we define the upper and lower traces
$u(+0,\cdot), u(-0,\cdot) \in H^{1/2}(\Sph^1)$
by
\begin{equation}
  u(\pm 0, \cdot)
  :=
  u|_{\{\pm z > 0\}}\big|_{\{z=0\}}
    \circ \theta.
\end{equation}
Note that
\begin{equation}
  \forall \phi \in \arcs{3}{1}
  \quad
  u(+0, \phi) = u(-0,\phi),
\end{equation}
but the value may jump across the slits.
Note also that the Rellich lemma holds on $\Sph^2_3$:
the inclusion $H^1(\Sph^2_3) \hookrightarrow L^2(\Sph^2_3)$
is compact.

For each $\alpha \in [0,1)$ and $j \in \Z_3$
we define the Sobolev spaces
\begin{equation}
\begin{gathered}
  H^1_{(j;3)}(\Sph^2_3)
    :=
    \{
      u \in H^1(\Sph^2_3)
      \st
      u \circ \rot_z^{-2\pi/3}
      =
      e^{2j\pi i/3}u
    \},
  \\
  H^1_{[\alpha]}(\Sph^2_3)
    :=
    \{
      u \in H^1(\Sph^2_3)
      \st
      u(+0,\cdot) = e^{2 \alpha \pi i}u(-0,\cdot)
        \mbox{ on } \arcs{3}{0}
    \},
  \\
  H^1_{[\alpha],j}(\Sph^2_3)
    :=
    H^1_{[\alpha]}(\Sph^2_3) \cap H^1_{(j;3)}(\Sph^2_3).
\end{gathered}
\end{equation}
Note that
$H^1_{[0]}(\Sph^2_3) = H^1(\Sph^2)$
and $H^1_{[0],j}(\Sph^2_3) = H^1_{(j;3)}(\Sph^2)$
for each $j \in \Z_3$.

\begin{lemma}
\label{slit-density}
For each $\alpha \in [0,1)$ and $j \in \Z_3$
the functions in
$
  H^1_{[\alpha],j}(\Sph^2_3)
  \cap
  C^\infty(\Sph^2_3)
$
that vanish on a neighborhood of the union $\pts{3}$
of the six slit endpoints
and whose restrictions to the northern and southern hemispheres
extend smoothly to the equator
are dense in $H^1_{[\alpha],j}(\Sph^2_3)$.
\end{lemma}

\begin{proof}
The operations proving \ref{junction-density} go through here too,
replacing $u|_{\{z<0\}}$ by
$e^{2 \alpha \pi i}u|_{\{z<0\}}$
when defining the mollification
of the truncated and logarithmically cut-off function $u$
on a neighborhood in $\{z>0\}$ of a slit,
and, symmetrically, replacing $u|_{\{z>0\}}$ by
$e^{-2 \alpha \pi i}u|_{\{z>0\}}$
when defining its mollification
on a neighborhood in $\{z<0\}$ of a slit.
\end{proof}

We will consider on $H^1(\Sph^2_3) \times H^1(\Sph^2_3)$
the Hermitian form
\begin{equation}
  \bop_{\Sph^2_3}(u,v)
  :=
  \langle du, dv \rangle_{L^2(\Sph^2_3)}
  -\langle u, V^{\Sph^2}_3 v \rangle_{L^2(\Sph^2_3)}.
\end{equation}

\begin{lemma}
\label{slit-removal}
For $\alpha \in [0,1)$ and $j \in \Z_3$,
if $u \in H^1_{[\alpha],j}(\Sph^2_3)$ satisfies
$\bop_{\Sph^2_3}(u,v) = 0$
for every $v \in H^1_{[\alpha],j}(\Sph^2_3)$
vanishing on a neighborhood of $\pts{3}$,
then $\bop_{\Sph^2_3}(u,v) = 0$ for every $v \in H^1_{[\alpha],j}(\Sph^2_3)$.
\end{lemma}

\begin{proof}
Arguing much as in \ref{puncture-removal},
we use the continuity of
$\bop_{\Sph^2_3}(u,\cdot)$ on $H^1(\Sph^2_3)$
and the density \ref{slit-density}.
\end{proof}

Next we define
for any $\alpha \in [0,1)$
and $j \in \Z_3$
\begin{equation}
\label{slit-ind-nul-def}
\begin{gathered}
  \begin{aligned}
  \ind_{\alpha,j}(\sphop_3, \Sph^2_3)
  :=
  \sum_{\lambda < 0}
  \dim
  \{
    &u \in H^1_{[\alpha],j}(\Sph^2_3)
    \st
    \forall v \in H^1_{[\alpha],j}(\Sph^2_3)
    \\
    &\;
    \bop_{\Sph^2_3}(u,v) = \lambda \langle u, v \rangle_{L^2(\Sph^2_3)}
  \},
  \end{aligned}
  \\
  \nul_{\alpha,j}(\sphop_3, \Sph^2_3)
  :=
  \dim
  \{
    u \in H^1_{[\alpha],j}(\Sph^2_3)
    \st
    \forall v \in H^1_{[\alpha],j}(\Sph^2_3)
    \;\,
    \bop_{\Sph^2_3}(u,v) = 0
  \}.
\end{gathered}
\end{equation}

\begin{lemma}
\label{slit-form-equiv}
For each $\alpha \in [0,1) \cap \Q$
and for each $j \in \Z_3$
\begin{equation*}
  \ind_{\alpha,j}(\ks_3)
  =
  \ind_{\alpha,j}(\sphop_3, \Sph^2_3),
  \qquad
  \nul_{\alpha,j}(\ks_3)
  =
  \nul_{\alpha,j}(\sphop_3,\Sph^2_3).
\end{equation*}
\end{lemma}

\begin{proof}
The map $\stereop^{-1} \circ \ewcover_3$
restricts to a conformal diffeomorphism
of $\ks_3^{(0,2\pi)}$ onto $\Sph^2_3$,
with
$
  (\stereop^{-1} \circ \ewcover_3)(\{z=\pi\} \cap \ks_3)
  =
  \{(0,\theta) \st \theta \in \arcs{3}{1}\}
$
and
$
  (\stereop^{-1} \circ \ewcover_3)(\{z=0\} \cap \ks_3)
  =
  (\stereop^{-1} \circ \ewcover_3)(\{z=2\pi\} \cap \ks_3)
  =
  \{(0,\theta) \st \theta \in \arcs{3}{0}\}
$.
The proof is completed along much the same lines
as that of \ref{index-nullity-conf-comp-equiv},
using  \ref{slit-density} and \ref{slit-removal}
in place of \ref{junction-density} and \ref{puncture-removal},
and using the fact that eigenfunctions of $\bop_{\Sph^2_3}$
in $H^1_{[n/m],~j}(\Sph^2_3)$
can be pulled back via $\stereop^{-1} \circ \ewcover_3$
and extended quasiperiodically
to elements of $C^\infty_{n/m, ~j}(\ks_3)$
(which descend to smooth functions on $\ks_{3,m}$,
and, as in \ref{index-nullity-conf-comp-equiv},
we appeal to \ref{jacobi_finiteness} for the nullity).
\end{proof}

\begin{proposition}
\label{k=3_bracket}
For each $\alpha \in (0,1) \cap \Q$ we have
\begin{equation*}
  \ind_{\alpha,0}(\ks_3) = 2
  \quad \mbox{and} \quad
  \nul_{\alpha,0}(\ks_3) = 0
\end{equation*}
and for each $j \in \{1,2\}$
\begin{equation*}
  \ind_{\alpha,j}(\ks_3) \geq 1
  \quad \mbox{and} \quad
  \ind_{\alpha,j}(\ks_3)
    +\nul_{\alpha,j}(\ks_3)
  \leq
  2.
\end{equation*}
\end{proposition}

\begin{proof}
We define for each $j \in \Z_3$
\begin{equation}
\begin{gathered}
  \begin{aligned}
  \ind_{\N,j}(\sphop_3, \Sph^2_3)
  :=
  \sum_{\lambda < 0}
  \dim
  \{
    &u \in H^1_{(j;3)}(\Sph^2_3)
    \st
    \forall v \in H^1_{(j;3)}(\Sph^2_3)
    \\
    &\;
    \bop_{\Sph^2_3}(u,v) = \lambda \langle u, v \rangle_{L^2(\Sph^2_3)}
  \},
  \end{aligned}
  \\
  \nul_{\N,j}(\sphop_3, \Sph^2_3)
  :=
  \dim
  \{
    u \in H^1_{(j;3)}(\Sph^2_3)
    \st
    \forall v \in H^1_{(j;3)}(\Sph^2_3)
    \;\,
    \bop_{\Sph^2_3}(u,v) = 0
  \},
\end{gathered}
\end{equation}
and also $\ind_{\D,j}(\sphop_3,\Sph^2_3)$
and $\nul_{\D,j}(\sphop_3,\Sph^2_3)$
by replacing each instance of $H^1_{(j;3)}(\Sph^2_3)$
by its intersection with $H^1_0(\Sph^2_3)$,
the closure in $H^1(\Sph^2_3)$
of the smooth functions with support disjoint
from $\partial\Sph^2_3 = \stereop^{-1}(e^{i\closure(\arcs{3}{0})})$.

By the min-max characterization of eigenvalues
we have for each $j \in \Z_3$ and $\alpha \in (0,1) \cap \Q$
\begin{equation}
\begin{gathered}
  \ind_{\D,j}(\sphop_3,\Sph^2_3)
    \leq
    \ind_{\alpha,j}(\sphop_3,\Sph^2_3)
    \leq
    \ind_{\N,j}(\sphop_3,\Sph^2_3)
  \\
  \begin{aligned}
  \ind_{\D,j}(\sphop_3,\Sph^2_3)
      +\nul_{\D,j}(\sphop_3,\Sph^2_3)
    &\leq
    \ind_{\alpha,j}(\sphop_3,\Sph^2_3)
      +\nul_{\alpha,j}(\sphop_3,\Sph^2_3)
    \\
    &\leq
    \ind_{\N,j}(\sphop_3,\Sph^2_3)
      +\nul_{\N,j}(\sphop_3,\Sph^2_3),
  \end{aligned}
\end{gathered}
\end{equation}
while, referring to
Sections \ref{fundamentals} and \ref{Zaremba},
we also have
\begin{equation}
\begin{gathered}
  \ind_{\D,j}(\sphop_3,\Sph^2_3)
    =
    \ind_{\D,j}(\ks_3^{(0,\pi)})
      +\ind_{\DN,j}(\ks_3^{(0,\pi)}),
  \\
  \ind_{\N,j}(\sphop_3,\Sph^2_3)
    =
    \ind_{\N,j}(\ks_3^{(0,\pi)})
      +\ind_{\DN,j}(\ks_3^{(0,\pi)}),
\end{gathered}
\end{equation}
and likewise with nullity in place of index.

In particular,
referring to \ref{D_and_N_half_period},
\ref{z-upper}, and \ref{Zclasses},
this establishes
\begin{equation}
\begin{gathered}
  \ind_{\alpha,0}(\ks_3) + \nul_{\alpha,0}(\ks_3) = 2,
  \\
  \forall j \in \Z_3
  \quad
  1 \leq \ind_{\alpha,j}(\ks_3) \leq 2,
  \\
  \forall j \in \{1,2\}
  \quad
  \ind_{\alpha,j}(\ks_3)
    +\nul_{\alpha,j}(\ks_3)
  \leq
  3
\end{gathered}
\end{equation}
for all $\alpha \in (0,1) \cap \Q$.
If $\nul_{\alpha,0}(\ks_3) \neq 0$,
then zero would be the second eigenvalue
of $-\sphop_3$ on $\Sph^2_3$
under both the Dirichlet condition
and the quasiperiodic boundary condition
with phase $e^{2\pi i \alpha} \neq 1$
across the slits
and the two problems would share a common such eigenfunction
(a Jacobi field),
but $\nul_{\D,0}(\sphop_3,\Sph^2_3) = 1$
is generated by
$2z/(1+z^2)$
(induced by vertical translations on $\ks_3$,
$z$ here being the height function on $\Sph^2 \subset \R^3$),
whose pullback to $\ks_3$ is periodic,
so has phase $1 \neq e^{2\pi i \alpha}$;
we conclude that $\nul_{\alpha,0}(\ks_3)=0$.
Similarly,
$\nul_{\N,1}(\sphop_3,\Sph^2_3) = 1$
and
$\nul_{\N,2}(\sphop_3,\Sph^2_3) = 1$
are generated by 
$(x \pm iy)^2/(1+z^2)$
(induced by horizontal translations on $\ks_3$,
$(x,y,z)$ here being the standard coordinates
on $\R^3 \supset \Sph^2$),
which also correspond to $2\pi$-periodic
functions on $\ks_3$,
so we conclude that
$-\sphop_3$ on $\Sph^2_3$
has strictly positive
third eigenvalue
under the quasiperiodic boundary condition
with phase $e^{2\pi i \alpha} \neq 1$,
completing the proof. 
\end{proof}

\subsection*{Reduction to the circle}

We define the sign pattern $\varsigma: \Sph^1 \to \{1,-1\}$,
the corresponding multiplier $\sgnmul: L^2(\Sph^1) \to L^2(\Sph^1)$
and, for each $\alpha \in (0,1)$,
the multiplier
$\mult_\alpha : L^2(\Sph^1) \to L^2(\Sph^1)$
by
\begin{equation}
\begin{gathered}
  \varsigma(\theta)
  :=
  \begin{cases}
    1 &\mbox{for } \theta \in \arcs{3}{1}
    \\
    -1 &\mbox{for } \theta \in \arcs{3}{0},
  \end{cases}
  \quad
  \sgnmul u := \varsigma u,
  \\
  \mult_\alpha u := e^{-i \varsigma \alpha \pi} u
    = (\cos \alpha \pi)\ident u - i(\sin \alpha \pi)\sgnmul u,
\end{gathered}
\end{equation}
where $\ident$ is the identity operator on $H^{-\infty}(\Sph^1)$.
Note that each $\mult_\alpha$
commutes with the (pullback) action
of $\theta \mapsto \theta-2\pi/3$.
For future reference we record the Fourier coefficients
\begin{equation}
\label{FT_varsigma}
  \widehat{\varsigma}(\ell)
  =
  \begin{cases}
    \dfrac{6(-1)^{(\abs{\ell}-3)/6}}{\pi\abs{\ell}}
      &\mbox{if } \ell \in 6\Z + 3,
    \\
    0
      &\mbox{otherwise.}
  \end{cases}
\end{equation}

On the circle we define the trace space
\begin{equation}
\begin{aligned}
  H^{1/2}_\varsigma(\Sph^1)
  &:=
  \{
    u \in H^{1/2}(\Sph^1)
    \st
    \varsigma u \in H^{1/2}(\Sph^1)
  \}
  \\
  &=
  H^{1/2}(\Sph^1) \cap \sgnmul H^{1/2}(\Sph^1),
\end{aligned}
\end{equation}
and we observe that for all $\alpha \in (0,1)$
\begin{equation}
\begin{aligned}
  H^{1/2}_\varsigma(\Sph^1)
  &=
  \{
    u \in H^{1/2}(\Sph^1)
    \st
    \mult_\alpha u \in H^{1/2}(\Sph^1)
  \}
  \\
  &=
  \{
    u(-0,\cdot)
    \st
    u \in H^1_{[\alpha]}(\Sph^2_3)
  \},
\end{aligned}
\end{equation}
while
\begin{equation}
  u(+0,\cdot)
  =
  e^{i\alpha\pi}\mult_\alpha u(-0,\cdot)
\end{equation}
for all $u \in H^1_{[\alpha]}(\Sph^2_3)$.
Note that $H^{1/2}_\varsigma(\Sph^1)$
is not a closed subspace of $H^{1/2}(\Sph^1)$;
it is instead a Hilbert space under the graph norm (for $\sgnmul$)
\begin{equation}
\label{sgnmul_norm}
  \norm{u}_{H^{1/2}_\varsigma(\Sph^1)}
  :=
  \sqrt{
  \norm{u}_{H^{1/2}(\Sph^1)}^2
    +\norm{\sgnmul u}_{H^{1/2}(\Sph^1)}^2
  }.
\end{equation}

We also define
\begin{equation}
  \dot{H}^{1/2}_{\varsigma}(\Sph^1)
  :=
  \{
    u \in \dot{H}^{1/2}(\Sph^1)
    \st
    \varsigma u \in \dot{H}^{1/2}(\Sph^1)
  \}
\end{equation}
and for each $j \in \{1,2\}$
\begin{equation}
  H^{1/2}_{\varsigma,(j;3)}(\Sph^1)
  :=
  H^{1/2}_\varsigma(\Sph^1) \cap H^{1/2}_{(j;3)}(\Sph^1)
  \subset
  \dot{H}^{1/2}_\varsigma(\Sph^1);
\end{equation}
we note that each $H^{1/2}_{\varsigma,(j;3)}$
is invariant under each $\mult_\alpha$.
Given also $\alpha \in (0,1)$, we further define
on
$
  \dot{H}^{1/2}_\varsigma(\Sph^1)
  \times
  \dot{H}^{1/2}_\varsigma(\Sph^1)
$
the Hermitian form
\begin{equation}
  \iop{\alpha}(u, v)
  :=
  \langle \nop_3 u, v \rangle_{L^2(\Sph^1)}
    +\langle \nop_3 \mult_\alpha u, \mult_\alpha v \rangle_{L^2(\Sph^1)}
\end{equation}
and for each $j \in \{1,2\}$ the index and nullity
\begin{equation}
\label{iop_ind_nul_def}
\begin{gathered}
  \ind_j(\iop{\alpha})
    :=
    \sup
    \{
      \dim X
      \st
      X \leq H^{1/2}_{\varsigma,(j;3)}(\Sph^1)
      \mbox{ and }
      \forall u \in X \setminus \{0\}
      \;\,
      \iop{\alpha}(u,u) < 0
    \},
  \\
  \nul_j(\iop{\alpha})
    :=
    \dim
    \{
      u \in H^{1/2}_{\varsigma,(j;3)}(\Sph^1)
      \st
      \forall v \in H^{1/2}_{\varsigma,(j;3)}(\Sph^1)
      \;\,
      \iop{\alpha}(u,v) = 0
    \}.
\end{gathered}
\end{equation}
Note that for each $j \in \{1,2\}$
\begin{equation}
  \ind_j(\iop{\alpha}) = \ind_{-j}(\iop{1-\alpha}),
  \quad
  \nul_j(\iop{\alpha}) = \nul_{-j}(\iop{1-\alpha}),
\end{equation}
as complex conjugation takes $H^{1/2}_{\varsigma,(j;3)}(\Sph^1)$
to $H^{1/2}_{\varsigma,(-j;3)}(\Sph^1)$,
carrying $\iop{\alpha}$ to $\iop{1-\alpha}$.

\begin{proposition}
\label{iop_ind_nul_equiv}
For each $\alpha \in (0,1) \cap \Q$ 
\begin{equation*}
\begin{gathered}
  \ind_{\alpha,1}(\ks_3)
    =
    \ind_{\alpha,2}(\ks_3)
    =
    \ind_1(\iop{\alpha})
    =
    \ind_2(\iop{\alpha}),
  \\
  \nul_{\alpha,1}(\ks_3)
    =
    \nul_{\alpha,2}(\ks_3)
    =
    \nul_1(\iop{\alpha})
    =
    \nul_2(\iop{\alpha}).
\end{gathered}
\end{equation*}
\end{proposition}

\begin{proof}
We have already observed the first equality on each line;
the final equality on each line follows similarly
from the fact that $\theta \mapsto -\theta$
takes $H^{1/2}_{\varsigma,(1;3)}(\Sph^1)$
to $H^{1/2}_{\varsigma,(2;3)}(\Sph^1)$
and preserves $\iop{\alpha}$.
In view of \ref{slit-form-equiv}
to complete the proof it suffices to establish
\begin{equation}
  \ind_{\alpha,1}(\sphop_3,\Sph^2_3)
  =
  \ind_1(\iop{\alpha}),
  \quad
  \nul_{\alpha,1}(\sphop_3,\Sph^2_3)
  =
  \nul_1(\iop{\alpha}).
\end{equation}

Arguing much as in the proofs of
\ref{nop-sphop} and \ref{z-reduction},
we observe that
for all $u \in H^1_{[\alpha],1}(\Sph^2_3)$
\begin{equation}
\label{slit_bop_blocks}
\begin{aligned}
  \bop_{\Sph^2_3}(u,u)
  &=
  \bop_3(u|_{\Sph^2_-}, u|_{\Sph^2_-})
    +\bop_3(\refl_{\{z=0\}}^*u|_{\Sph^2_+},\refl_{\{z=0\}}^*u|_{\Sph^2_+})
  \\
  &=
  \iop{\alpha}(u(-0,\cdot),u(-0,\cdot))
  \\
    &\quad
    +\bop_3
      (
        u|_{\Sph^2_-}-\jext_3u(-0,\cdot), ~
        u|_{\Sph^2_-}-\jext_3u(-0,\cdot)
      )
  \\
    &\quad
    +\bop_3
      (
        \refl_{\{z=0\}}^*u|_{\Sph^2_+}-\jext_3u(+0,\cdot), ~
        \refl_{\{z=0\}}^*u|_{\Sph^2_+}-\jext_3u(+0,\cdot)
      )
\end{aligned}
\end{equation}
recalling in particular \ref{DtN}.
According to \ref{D_and_N_half_period},
the first Dirichlet eigenvalue of $\sphop_3$
on $H^1_{(1;3)}(\Sph^2_-)$ is strictly positive,
and so
the underlying decomposition
\begin{equation}
\begin{gathered}
  H^1_{[\alpha],1}(\Sph^2_3)
  \to
    H^{1/2}_{\varsigma,(1;3)}(\Sph^1)
    \oplus
    H^1_{0,(1;3)}(\Sph^2_-)
    \oplus
    H^1_{0,(1;3)}(\Sph^2_-)
  \\
  u
  \mapsto
    (
      u(-0,\cdot), ~
      u|_{\Sph^2_-} - \jext_3u(-0,\cdot), ~
      \refl_{\{z=0\}}^*u|_{\Sph^2_+} - \jext_3u(+0,\cdot)
    )
\end{gathered}
\end{equation}
is an isomorphism
and each of the last two blocks of  
\ref{slit_bop_blocks}
is positive definite,
ending the proof.
\end{proof}

\subsection*{Frequency parity}

To complete the evaluation of
$\ind_\alpha(\ks_3)$
and $\nul_\alpha(\ks_3)$,
for each $\alpha \in (0,1) \cap \Q$,
it suffices to compute
$\ind_2(\iop{\alpha})$ and $\nul_2(\iop{\alpha})$.
We define
\begin{equation}
\begin{gathered}
  H^s_{(+,2;3)}(\Sph^1)
    :=
    \{
      u \in H^s(\Sph^1)
      \st
      \supp \widehat{u} \subseteq 1+6\Z
    \}
    \quad (s \in \R),
  \\
  H^s_{(-,2;3)}(\Sph^1)
    :=
     \{
      u \in H^s(\Sph^1)
      \st
      \supp \widehat{u} \subseteq 4+6\Z
    \}
    \quad (s \in \R),
\end{gathered}
\end{equation}
the two eigenspaces of $\theta \mapsto \theta - \pi/3$
in $H^s_{(2;3)}(\Sph^1)$,
and
\begin{equation}
  H^{1/2}_{\varsigma,(\pm,2;3)}(\Sph^1)
    :=
    H^{1/2}_\varsigma(\Sph^1)
    \cap
    H^{1/2}_{(\pm,2;3)}(\Sph^1),
\end{equation}
so that
\begin{equation}
\begin{gathered}
  H^s_{(2;3)}(\Sph^1)
  =
  H^s_{(+,2;3)} \oplus H^s_{(-,2;3)}
  \quad
  (s \in \R),
  \\
  \sgnmul H^0_{(\pm,2;3)}(\Sph^1)
  =
  H^0_{(\mp,2;3)}(\Sph^1),
  \quad
  \sgnmul H^{1/2}_{\varsigma,(\pm,2;3)}(\Sph^1)
  =
  H^{1/2}_{\varsigma,(\mp,2;3)}(\Sph^1),
\end{gathered}
\end{equation}
and we define the Fourier multiplier
$\parity: H^{-\infty}_{(2;3)}(\Sph^1) \to H^{-\infty}_{(2;3)}(\Sph^1)$ by
\begin{equation}
  \widehat{\parity u}(\ell)
  :=
  \begin{cases}
    \widehat{u}(\ell) &\mbox{for } \ell \in 1+6\Z
    \\
    -\widehat{u}(\ell) &\mbox{for } \ell \in 4+6\Z.
  \end{cases}
\end{equation}

Note that
\begin{equation}
\begin{aligned}
  \parity \sgnmul = -\sgnmul \parity
  \quad &\mbox{on } H^0_{(2;3)}(\Sph^1),
  \\
  \parity \nop_3 = \nop_3 \parity
  \quad &\mbox{on } H^{1/2}_{(2;3)}(\Sph^1),
\end{aligned}
\end{equation}
and
for each $\alpha \in (0,1)$
\begin{equation}
\begin{aligned}
  \parity \mult_\alpha
  &=
  \parity((\cos \alpha\pi)\ident - i(\sin \alpha\pi)\sgnmul) 
  \\
  &=
  ((\cos \alpha\pi)\ident + i(\sin \alpha \pi)\sgnmul)\parity
  \\
  &=
  -\mult_{1-\alpha} \parity.
\end{aligned}
\end{equation}

Next we define
$\nop_3^+ := \nop_3\big|_{H^{1/2}_{\varsigma,(+,2;3)}(\Sph^1)}$
and
$\nop_3^- := \nop_3\big|_{H^{1/2}_{\varsigma,(-,2;3)}(\Sph^1)}$,
so that
\begin{equation}
\label{DtN_3pm-momentum}
\begin{gathered}
  \widehat{\nop_3^+ u}(\ell)
  =
  \begin{cases}
    \sigma_3(\ell)\widehat{u}(\ell)
      &\mbox{if } \ell \in 1 + 6\Z
    \\
    0
      &\mbox{otherwise,}
  \end{cases}
  \\
  \widehat{\nop_3^- v}(\ell)
  =
  \begin{cases}
    \sigma_3(\ell)\widehat{v}(\ell)
      &\mbox{if } \ell \in 4 + 6\Z
      \\
    0
      &\mbox{otherwise,}
  \end{cases}
  \\
  \text{where, recalling \ref{DtN-momentum},}
  \quad
  \sigma_3(\ell)
  :=
  \frac{\ell^2-4}{\abs{\ell}}.
\end{gathered}
\end{equation}

We also define, for each $s>0$,
on
$
  H^{1/2}_{\varsigma,(+,2;3)}(\Sph^1)
  \times
  H^{1/2}_{\varsigma,(+,2;3)}(\Sph^1)
$
the Hermitian form
\begin{equation}
  \qop_s(u,v)
  :=
  \langle
    \nop_3^+u,
    v
  \rangle_{L^2(\Sph^1)}
  +
  s^2
  \langle
    \nop_3^-\sgnmul u,
    \sgnmul v
  \rangle_{L^2(\Sph^1)}
\end{equation}
and the index and nullity
\begin{equation}
\begin{gathered}
  \ind(\qop_s)
    :=
    \sup
    \{
      \dim X
      \st
      X \leq H^{1/2}_{\varsigma,(+,2;3)}(\Sph^1)
      \mbox{ and }
      \forall u \in X \setminus \{0\}
      \;\,
      \qop_s(u,u) < 0
    \},
  \\
  \nul(\qop_s)
    :=
    \dim
    \{
      u \in H^{1/2}_{\varsigma,(+,2;3)}(\Sph^1)
      \st
      \forall v \in H^{1/2}_{\varsigma,(+,2;3)}(\Sph^1)
      \;\,
      \qop_s(u,v) = 0
    \}.
\end{gathered}
\end{equation}
Shortly, in \ref{iop-qop_equiv},
we will see that our problem reduces
to the determination of the index and nullity of
$\qop_s$.

\begin{lemma}
\label{qop_ind_plus_nul_above}
For each $s > 0$
$
  \ind(\qop_s) + \nul(\qop_s)
  \leq
  1
$.
\end{lemma}

\begin{proof}
If $X \leq H^{1/2}_{\varsigma,(+,2;3)}(\Sph^1)$
has $\dim X = 2$,
then there is a nonzero $u \in X$
with $\widehat{u}(1) = 0$,
and so, making use of \ref{DtN_3pm-momentum},
\begin{equation}
  \qop_s(u,u)
  \geq
  \langle
    \nop_3^+u,
    u
  \rangle_{L^2(\Sph^1)}
  \geq
  \sigma_3(-5)\norm{u}_{L^2(\Sph^1)}^2
  >
  0.
\end{equation}
\end{proof}

Now, for each $\alpha \in (0,1)$ we define the operator
\begin{equation}
  \rop_\alpha
  :=
  \mult_{1-\alpha} \parity
    \Big|_{H^{1/2}_{\varsigma,(2;3)}(\Sph^1)}
\end{equation}
(induced by the action of the symmetry $\rot_z^{\pi/3}\trans_z^\pi$
on the original surface),
whose eigenspaces we will use to split $\iop{\alpha}$
and accomplish the reduction to $\qop_s$.

\begin{lemma}
\label{rop_props}
For each $\alpha \in (0,1)$
$\rop_\alpha$ is an involution
on $H^{1/2}_{\varsigma,(2;3)}(\Sph^1)$,
and for all $u,v \in H^{1/2}_{\varsigma,(2;3)}(\Sph^1)$
\begin{equation*}
\begin{gathered}
  \langle \rop_\alpha u, \rop_\alpha v \rangle_{L^2(\Sph^1)}
    =
  \langle u, v \rangle_{L^2(\Sph^1)},
  \\
  \iop{\alpha}(\rop_\alpha u, \rop_\alpha v)
    =
    \iop{\alpha}(u,v).
\end{gathered}
\end{equation*}
\end{lemma}

\begin{proof}
First we compute
\begin{equation}
  \rop_\alpha^2
  =
  \mult_{1-\alpha} \parity \, \mult_{1-\alpha}\parity
  =
  \mult_{1-\alpha} (-\mult_\alpha \parity) \parity 
  =
  \ident.
\end{equation}
Next,
for any $u,v \in H^{1/2}_{\varsigma,(2;3)}(\Sph^1)$
\begin{equation}
  \langle \rop_\alpha u, \rop_\alpha v \rangle_{L^2(\Sph^1)}
  =
  \langle \mult_{1-\alpha}\parity u, \mult_{1-\alpha}\parity v \rangle_{L^2(\Sph^1)}
  =
  \langle \parity u, \parity v \rangle_{L^2(\Sph^1)}
  =
  \langle u, v \rangle_{L^2(\Sph^1)},
\end{equation}
and finally
\begin{equation}
\begin{aligned}
  \iop{\alpha}(\rop_\alpha u, \rop_\alpha v)
  &=
  \langle \nop_3 \rop_\alpha u, \rop_\alpha v \rangle_{L^2(\Sph^1)}
    +\langle
      \nop_3\mult_\alpha \rop_\alpha u,
      \mult_\alpha \rop_\alpha v
    \rangle_{L^2(\Sph^1)}
  \\
  &=
  \langle
    \nop_3 \mult_{1-\alpha} \parity u, ~
    \mult_{1-\alpha} \parity v
  \rangle_{L^2(\Sph^1)}
  \\
  &\qquad
  +\langle
    \nop_3 \mult_\alpha \mult_{1-\alpha} \parity u, ~
    \mult_\alpha \mult_{1-\alpha} \parity v
  \rangle_{L^2(\Sph^1)}
  \\
  &=
  \langle
    \parity \nop_3 \mult_\alpha u,
    \parity \mult_\alpha v
  \rangle_{L^2(\Sph^1)}
  +
  \langle
    \parity \nop_3 u,
    \parity v
  \rangle_{L^2(\Sph^1)}
  \\
  &=
  \langle
    \nop_3 \mult_\alpha u,
    \mult_\alpha v
  \rangle_{L^2(\Sph^1)}
  +
  \langle
    \nop_3 u,
    v
  \rangle_{L^2(\Sph^1)}
  \\
  &=
  \iop{\alpha}(u,v).
\end{aligned}
\end{equation}
\end{proof}

In particular,
for each $\alpha \in (0,1)$,
we have the $L^2(\Sph^1)$-orthogonal direct sum decomposition
\begin{equation}
\label{rop_decomp}
  H^{1/2}_{\varsigma,(2;3)}(\Sph^1)
  =
  \ker (\rop_\alpha-\ident)
    \oplus
    \ker (\rop_\alpha+\ident).
\end{equation}

We define
$\ind_2^\pm(\iop{\alpha})$
and $\nul_2^\pm(\iop{\alpha})$
(four definitions in total)
by replacing
$H^{1/2}_{\varsigma,(2;3)}(\Sph^1)$
in \ref{iop_ind_nul_def}
by $\ker (\rop_\alpha \mp \ident)$.

\begin{lemma}
\label{iop_rop_decomp}
For each $\alpha \in (0,1)$
\begin{equation*}
\begin{gathered}
  \ind_2(\iop{\alpha})
  =
  \ind_2^+(\iop{\alpha}) + \ind_2^-(\iop{\alpha}),
  \\
  \nul_2(\iop{\alpha})
  =
  \nul_2^+(\iop{\alpha}) + \nul_2^-(\iop{\alpha}).
\end{gathered}
\end{equation*}
\end{lemma}

\begin{proof}
From the final conclusion of \ref{rop_props}
we see that $\iop{\alpha}$ is block diagonal
with respect to \ref{rop_decomp}.
The nullity equality follows immediately,
as does the index relation
with $\geq$ in place of $=$.
To verify the reverse inequality
let $X \leq H^{1/2}_{\varsigma,(2;3)}(\Sph^1)$
be a finite-dimensional subspace
on which $\iop{\alpha}$ is negative definite.
We then obtain that the right-hand side is at least $\dim X$
by applying, as justified by \ref{rop_props},
the finite-dimensional spectral theorem
to the Hermitian linear map associated
via $L^2(\Sph^1)$
to $\iop{\alpha}$
restricted to $X + \rop_\alpha X$,
noting that this map has $\rop_\alpha$-invariant domain
and commutes with $\rop_\alpha$.
As this inequality holds for all such $X$, the proof is complete.
\end{proof}

Given any $s \in \R$ and $u \in H^s_{(2;3)}(\Sph^1)$,
we write
\begin{equation}
  u^{\parity,\pm}
  :=
  \frac{1}{2}(\ident \pm \parity)u
  \in
  H^s_{(\pm,2;3)}(\Sph^1)
\end{equation}
for its projections onto the $\pm 1$-eigenspaces of $\parity$.

\begin{proposition}
\label{iop-qop_equiv}
For each $\alpha \in (0,1)$
\begin{enumerate}
\item
  \label{rop_parity_plus_iso}
  each of the two maps
  $
    \ker (\rop_\alpha \pm \ident)
    \ni
    u
    \mapsto
    u^{\parity,+}
    \in
    H^{1/2}_{\varsigma,{(+,2;3)}}(\Sph^1)
  $
  is an isomorphism,
  
  \item
  \label{iop-qop_rop_plus}
  for all $u \in \ker (\rop_\alpha + \ident)$
  we have
  $
    \iop{\alpha}(u,u)
      =
      2\qop_{\tan \tfrac{\alpha\pi}{2}}(u^{\parity,+},u^{\parity,+})
  $,

\item
\label{iop-qop_rop_minus}
  for all $u \in \ker (\rop_\alpha - \ident)$
  we have
  $
    \iop{\alpha}(u,u)
      =
      2\qop_{\cot \tfrac{\alpha\pi}{2}}(u^{\parity,+},u^{\parity,+})
  $,

  \item
  \label{iop-qop_ind}
  $
    \ind_2(\iop{\alpha})
    =
    \ind(\qop_{\tan (\alpha\pi/2)})
      +\ind(\qop_{\cot (\alpha\pi/2)})
  $,
  and

  \item
  \label{iop-qop_nul}
  $
    \nul_2(\iop{\alpha})
    =
    \nul(\qop_{\tan (\alpha\pi/2)})
      +\nul(\qop_{\cot (\alpha\pi/2)})
  $.
\end{enumerate}
\end{proposition}

\begin{proof}
To see surjectivity for the $+$ case of
\ref{rop_parity_plus_iso}
suppose
$v_+ \in H^{1/2}_{\varsigma,(+,2;3)}(\Sph^1)$
and define
\begin{equation}
  v_-
  :=
  \frac{i \sin \alpha\pi}{1 + \cos \alpha\pi}
    \sgnmul v_+
  \quad \mbox{and} \quad
  u
  :=
  v_+ + v_-;
\end{equation}
then $u \in H^{1/2}_{\varsigma,(2;3)}(\Sph^1)$,
$u^{\parity,\pm} = v_\pm$, and
\begin{equation}
\begin{aligned}
  \rop_\alpha u
  &=
  \mult_{1-\alpha}(v_+ - v_-)
  \\
  &=
  -(\cos \alpha\pi)v_+
    -i(\sin \alpha\pi)\sgnmul v_+
    +i \frac{\sin \alpha\pi \cos \alpha\pi}
      {1 + \cos \alpha\pi}
        \sgnmul v_+
    - \frac{\sin^2 \alpha\pi}
      {1 + \cos \alpha\pi}
        v_+
  \\
  &=
  -v_+ - i \frac{\sin \alpha\pi}{1+\cos \alpha\pi}\sgnmul v_+
  =
  -v_+ - v_- = -u.
\end{aligned}
\end{equation}

To confirm surjectivity in the $-$ case of
\ref{rop_parity_plus_iso}
we modify the definition of $v_-$ to
\begin{equation}
  v_-
  :=
  \frac{-i \sin \alpha\pi}{1 - \cos \alpha\pi}
    \sgnmul v_+
\end{equation}
and then verify $\rop_\alpha (v_+ + v_-) = v_+ + v_-$.
To check injectivity in both cases
suppose $u \in H^{1/2}_{\varsigma,(-,2;3)}(\Sph^1)$
satisfies $\rop_\alpha u = \pm u$;
then 
$
  \cos (\alpha\pi)u + i\sin(\alpha \pi) \sgnmul u
  =
  \pm u
$,
but $\sgnmul u \in H^{1/2}_{\varsigma,(+,2;3)}(\Sph^1)$
is orthogonal to $u$,
and $\sin \alpha \pi \neq 0$,
so in fact $u=0$.

Next, for any $u \in \ker(\rop_\alpha \pm \ident)$
we have according to the above
\begin{equation}
  \mult_\alpha u
  =
  \pm \parity u
  \quad \mbox{and} \quad
  u^{\parity,-}
  =
  \frac{\pm i \sin \alpha\pi}{1 \pm \cos \alpha\pi}\sgnmul u^{\parity,+},
\end{equation}
and so
\begin{equation}
\begin{aligned}
  \iop{\alpha}(u,u)
  &=
    \langle
      \nop_3^+ u^{\parity,+}, u^{\parity,+}
    \rangle_{L^2(\Sph^1)}
    \\
    &\quad
    +
    \frac{\sin^2 \alpha\pi}{(1 \pm \cos \alpha\pi)^2}
    \langle
      \nop_3^- \sgnmul u^{\parity,+}, \sgnmul u^{\parity,+}
    \rangle_{L^2(\Sph^1)}
    +
    \langle
      \nop_3 \parity u, \parity u 
    \rangle_{L^2(\Sph^1)}
  \\
  &=
  2\qop_{\tfrac{\sin \alpha\pi}{1 \pm \cos \alpha\pi}}(u^{\parity,+},u^{\parity,+}),
\end{aligned}
\end{equation}
completing the proof
of \ref{iop-qop_rop_plus} and \ref{iop-qop_rop_minus},
in view of the half-angle identities for the sine and cosine functions.

Items \ref{iop-qop_ind} and \ref{iop-qop_nul}
now follow from the preceding ones
in conjunction with \ref{iop_rop_decomp}.
\end{proof}

\subsection*{Threshold}

Note that $\qop_s$ is bounded
on
$
  H^{1/2}_{\varsigma,(+,2;3)}(\Sph^1)
  \times
  H^{1/2}_{\varsigma,(+,2;3)}(\Sph^1)
$
(using the norm \ref{sgnmul_norm})
and by \ref{DtN_3pm-momentum},
for all $u \in H^{1/2}_{\varsigma,(+,2;3)}(\Sph^1)$
\begin{equation}
  \qop_s(u,u) + 5(1+s^2)\norm{u}_{L^2(\Sph^1)}^2
  \geq
  \min\{1,s^2\}\norm{u}_{H^{1/2}_\varsigma(\Sph^1)}^2.
\end{equation}
Note also that the inclusion
$H^{1/2}_{\varsigma,(+,2;3)}(\Sph^1) \hookrightarrow L^2(\Sph^1)$
is compact.
We define
\begin{equation}
  \gamma(s)
  :=
  \inf_{0 \neq u \in H^{1/2}_{\varsigma,(+,2;3)}(\Sph^1)}
  \frac{\qop_s(u,u)}{\norm{u}_{L^2(\Sph^1)}^2}
  \in \R
\end{equation}
to be the least eigenvalue of the Hermitian form $\qop_s$
relative to the $L^2(\Sph^1)$ inner product
(where $\lambda \in \R$ is an eigenvalue of $\qop_s$
if there exists a nonzero $u \in H^{1/2}_{\varsigma,(+,2;3)}(\Sph^1)$
such that $\qop_s(v,u) = \langle v, \lambda u \rangle_{L^2(\Sph^1)}$
for all $v \in H^{1/2}_{\varsigma,(+,2;3)}(\Sph^1)$).

\begin{lemma}
\label{single_crossing}
The function $\gamma$ is continuous and strictly increasing on $(0,\infty)$.
\end{lemma}

\begin{proof}
For each $u \neq 0$ the map
$s^2 \mapsto \qop_s(u,u)/\norm{u}_{L^2(\Sph^1)}^2$
is affine, so concave;
infimizing, we conclude that
$\gamma(\sqrt{\tau})$
is a concave function of $\tau$,
and in particular $\gamma$ itself is continuous.
Now let $s_2 > s_1 > 0$
and let $u$ be a first eigenfunction of $\qop_{s_2}$
with $\norm{u}_{L^2(\Sph^1)} = 1$,
so that $\gamma(s_2) = \qop_{s_2}(u,u)$.
By \ref{DtN_3pm-momentum}
we have
$
  \langle
    \nop_3^- \sgnmul u,
    \sgnmul u
  \rangle_{L^2(\Sph^1)}
  \geq 0
$
with equality only if
$\varsigma u \in \C e^{-2i\theta}$,
but 
(as can be seen with the aid of \ref{FT_varsigma})
$\varsigma(\theta)e^{-2i\theta} \not \in H^{1/2}(\Sph^1)$.
Consequently,
\begin{equation*}
  \gamma(s_1)
  \leq
  \qop_{s_1}(u,u)
  =
  \qop_{s_2}(u,u)
    -(s_2^2 - s_1^2)
      \langle
        \nop_3^- \sgnmul u,
        \sgnmul u
      \rangle_{L^2(\Sph^1)}
  < 
  \gamma(s_2).
\end{equation*}
\end{proof}

\begin{lemma}
\label{threshold_lower_bound}
We have $\gamma(1) < 0$.
\end{lemma}

\begin{proof}
From Theorem \ref{m=2}
and \ref{k=3_decomp} with $m = 2$
\begin{equation*}
  9
  =
  \ind(\ks_{3,2})
  =
  \ind_0(\ks_3) + \ind_{\tfrac{1}{2}}(\ks_3)
  =
  3
    +\ind_{\tfrac{1}{2}}(\ks_3),
\end{equation*}
whence $\ind_{1/2}(\ks_3)=6$.
Using \ref{j_decomp_given_alpha},
\ref{j_sym},
and \ref{k=3_bracket},
this implies
$\ind_{1/2,~2}(\ks_3) = 2$.
Finally,
applying
\ref{iop_ind_nul_equiv}
and \ref{iop-qop_equiv},
we obtain
$\ind(\qop_1) = 1$.
In particular
the least eigenvalue $\gamma(1)$ of $\qop_1$ is strictly negative.
\end{proof}

\begin{lemma}
\label{threshold_upper_bound}
We have $\gamma(\sqrt{3}) > 0$.
\end{lemma}

\begin{proof}
Let $u \in H^{1/2}_{\varsigma,(+,2;3)}(\Sph^1)$.
Referring to \ref{DtN_3pm-momentum},
we have
\begin{equation}
  \langle \nop_3^- \sgnmul u, \sgnmul u \rangle_{L^2(\Sph^1)}
  \geq
  \sigma_3(4)
    \big(
      \norm{u}_{L^2(\Sph^1)}^2
      -2\pi\abs{\widehat{\sgnmul u}(-2)}^2
    \big),
\end{equation}
and so
\begin{equation}
\label{qop_upper_prelim}
\begin{aligned}
  \qop_{\sqrt{3}}(u,u)
  &=
  \langle \nop_3^+ u, u \rangle_{L^2(\Sph^1)}
    +3\langle \nop_3^- \sgnmul u, \sgnmul u\rangle_{L^2(\Sph^1)}
  \\
  &\geq
  \langle (\nop_3^+ + 9)u, u \rangle_{L^2(\Sph^1)}
    -18\pi\abs{\widehat{\sgnmul u}(-2)}^2.
\end{aligned}
\end{equation}

On the other hand,
\begin{equation}
\label{bound_in_S}
\begin{aligned}
  18\pi\abs{\widehat{\sgnmul u}(-2)}^2
  &=
  \frac{9}{2\pi}\abs{\langle \varsigma e^{-2i\cdot}, u \rangle_{L^2(\Sph^1)}}^2
  =
  18\pi\abs*{
    \sum_{\ell \in 1 + 6\Z}
      \widehat{\varsigma}(\ell+2) \widehat{u}(\ell)
  }^2
  \\
  &=
  18\pi\abs*{
    \sum_{\ell \in 1 + 6\Z}
      \frac{\widehat{\varsigma}(\ell+2)}{\sqrt{\sigma_3(\ell)+9}}
      \,
      \sqrt{\sigma_3(\ell)+9} \, \widehat{u}(\ell)
  }^2
  \\
  &\leq
  S
  \langle
    u, ~
    (\nop_3^+ + 9)u
  \rangle_{L^2(\Sph^1)},
\end{aligned}
\end{equation}
where
\begin{equation}
  S
  :=
  \frac{9}{2\pi}
  \langle
    \sgnmul e^{-2i\cdot}, ~
    (\nop_3^+ + 9)^{-1}\sgnmul e^{-2i\cdot}
  \rangle_{L^2(\Sph^1)},
\end{equation}
having used the Cauchy--Schwarz inequality
and the fact,
evident from \ref{DtN_3pm-momentum},
that $\sigma_3(\ell)+9 > 0$
for all $\ell \in 1+6\Z$.

Now assume that $u$ is a first eigenfunction of $\qop_{\sqrt{3}}$
with $\norm{u}_{L^2(\Sph^1)}^2 = 1$.
Then, applying \ref{bound_in_S} to \ref{qop_upper_prelim},
we obtain
\begin{equation}
  \gamma(\sqrt{3})
  =
  \qop_{\sqrt{3}}(u,u) 
  \geq
  (1-S)\langle (\nop_3^+ + 9)u, ~ u \rangle_{L^2(\Sph^1)},
\end{equation}
and so to complete the proof it suffices to show $S < 1$
(using again the fact that $\sigma_3 + 9$
is strictly positive on $1+6\Z$).

Using \ref{FT_varsigma} and again \ref{DtN_3pm-momentum},
we compute
\begin{equation*}
  \frac{\pi^2}{36}S
  =
  \frac{\pi^2}{4}\sum_{\ell \in 1 + 6\Z}
    \frac{\abs{\widehat{\varsigma}(\ell+2)}^2}{\sigma_3(\ell) + 9}
  =
  \sum_{\ell \in 1 + 6\Z}
    \frac{9\abs{\ell}}{(\ell+2)^2(\ell^2+9\abs{\ell}-4)}.
\end{equation*}
On the other hand,
since $\varsigma(\theta) \, e^{-2i\theta}$
has $L^2(\Sph^1)$ norm $\sqrt{2\pi}$,
using Parseval's identity and again \ref{FT_varsigma},
we have
\begin{equation*}
  \frac{\pi^2}{36}
  =
  \sum_{\ell \in 1 + 6\Z}
    \frac{1}{(\ell+2)^2},
\end{equation*}
and therefore
\begin{equation*}
\begin{aligned}
  \frac{\pi^2}{36}(1-S)
  &=
  \sum_{\ell \in 1 + 6\Z}
    \frac{\ell^2-4}{(\ell+2)^2(\ell^2+9\abs{\ell}-4)}
  =
  \sum_{\ell \in 1 + 6\Z}
    \frac{\ell-2}{(\ell+2)[\abs{\ell}(\abs{\ell}+9)-4]}
  \\
  &=
  \sum_{n \in \Z}
    \frac{6n-1}{(6n+3)[\abs{6n+1}(\abs{6n+1}+9)-4]}
  \\
  &>
  \sum_{n=-4}^3 \frac{6n-1}{(6n+3)[\abs{6n+1}(\abs{6n+1} + 9)-4]}
  \\
  &=
  \frac{25}{21[23(32)-4]}
    +\frac{19}{15[17(26)-4]}
    +\frac{13}{9[11(20)-4]}
    \\
    &\quad
    +\frac{7}{3[5(14)-4]}
    -\frac{1}{3[1(10)-4]}
    +\frac{5}{9[7(16)-4]}
    \\
    &\quad
    +\frac{11}{15[13(22)-4]}
    +\frac{17}{21[19(28)-4]}
  \\
  &=
  \frac{25}{15372}
    +\frac{19}{6570}
    +\frac{13}{1944}
    +\frac{7}{198}
    -\frac{1}{18}
    +\frac{5}{972}
    +\frac{11}{4230}
    +\frac{17}{11088}
  \\
  &=
  \frac{88090307}{313283512080}
  >
  0,
\end{aligned}
\end{equation*}
where we have used the fact that on $1+6\Z$
the function
$\ell^2+9\abs{\ell}-4$ is globally strictly positive,
while the function
$(\ell-2)/(\ell+2)$ is strictly positive
except at $\ell=1$.
\end{proof}

It follows from \ref{single_crossing},
\ref{threshold_lower_bound},
and \ref{threshold_upper_bound}
that there exists a unique
$s^* \in (0,\infty)$ such that
$\gamma(s^*) = 0$,
and moreover $1 < s^* < \sqrt{3}$.
We define in turn
\begin{equation}
  \alpha^*
  :=
  \frac{2}{\pi} \arctan \frac{1}{s^*}
  \in
  \Big(\frac{1}{3} , ~ \frac{1}{2}\Big).
\end{equation}

\begin{proposition}
\label{ind_null_rel_threshold}
For all $\alpha \in (0,1) \cap \Q$
\begin{equation*}
\begin{aligned}
  \ind_\alpha(\ks_3)
  &=
  \begin{cases}
    4
      &\mbox{if }
      \alpha \in (0,\alpha^*] \cup [1-\alpha^*, 1)
    \\
    6 &\mbox{if }
      \alpha \in (\alpha^*, 1-\alpha^*),
  \end{cases}
  \\
  \nul_\alpha(\ks_3)
  &=
  \begin{cases}
    0 &\mbox{if }
      \alpha \not \in \{\alpha^*, 1-\alpha^*\}
    \\
    2 &\mbox{if } 
      \alpha \in \{\alpha^*, 1-\alpha^*\}.
  \end{cases}
\end{aligned}
\end{equation*}
\end{proposition}

\begin{proof}
Let $\alpha \in (0,1)$ be rational.
By \ref{j_decomp_given_alpha}, \ref{j_sym}
and \ref{k=3_bracket}
we must show
\begin{equation}
\begin{aligned}
  \ind_{\alpha,2}(\ks_3)
  &=
  \begin{cases}
    1
      &\mbox{if }
      \alpha \in (0,\alpha^*] \cup [1-\alpha^*, 1)
    \\
    2 &\mbox{if }
      \alpha \in (\alpha^*, 1-\alpha^*),
  \end{cases}
  \\
  \nul_{\alpha,2}(\ks_3)
  &=
  \begin{cases}
    0 &\mbox{if }
      \alpha \not \in \{\alpha^*, 1-\alpha^*\}
    \\
    1 &\mbox{if }
      \alpha \in \{\alpha^*, 1-\alpha^*\}.
  \end{cases}
\end{aligned}
\end{equation}

By \ref{iop_ind_nul_equiv}
and \ref{iop-qop_equiv}
we must equivalently show
\begin{equation}
\begin{aligned}
  \ind(\qop_{\tan (\alpha\pi/2)})
    +\ind(\qop_{\cot (\alpha\pi/2)})
  &=
  \begin{cases}
    1
      &\mbox{if }
      \alpha \in (0,\alpha^*] \cup [1-\alpha^*, 1)
    \\
    2 &\mbox{if }
      \alpha \in (\alpha^*, 1-\alpha^*),
  \end{cases}
  \\
  \nul(\qop_{\tan (\alpha\pi/2)})
    +\nul(\qop_{\cot (\alpha\pi/2)})
  &=
  \begin{cases}
    0 &\mbox{if }
      \alpha \not \in \{\alpha^*, 1-\alpha^*\}
    \\
    1 &\mbox{if }
      \alpha \in \{\alpha^*, 1-\alpha^*\}.
  \end{cases}
\end{aligned}
\end{equation}

By \ref{qop_ind_plus_nul_above}
we have
$\ind(\qop_s) + \nul(\qop_s) \leq 1$
for all $s > 0$.
Since $\gamma(s)$ is the least eigenvalue of $\qop_s$
and since $\gamma(s)$ is strictly increasing,
we conclude that
\begin{equation}
  \nul(\qop_s)
  =
  \begin{cases}
    1 &\mbox{if } s = s^*
    \\
    0 &\mbox{otherwise},
  \end{cases}
  \qquad
  \ind(\qop_s)
  =
  \begin{cases}
    1 &\mbox{if } s < s^*
    \\
    0 &\mbox{otherwise}.
  \end{cases}
\end{equation}
Since $s^* > 1$,
we have
$s \geq s^* \Rightarrow s^{-1} < s^*$,
and so it follows that
\begin{equation}
\begin{aligned}
  \ind(\qop_s) + \ind(\qop_{s^{-1}})
  &=
  \begin{cases}
    1 &\mbox{if} \quad s \geq s^* \mbox{ or } 1/s \geq s^*
    \\
    2 &\mbox{if} \quad 1/s^* < s < s^*,
  \end{cases}
  \\
  \nul(\qop_s) + \nul(\qop_{s^{-1}})
  &=
  \begin{cases}
    0 &\mbox{if} \quad s \not\in \{s^*, 1/s^*\} 
    \\
    1 &\mbox{if} \quad s \in \{s^*, 1/s^*\}.
  \end{cases}
\end{aligned}
\end{equation}

Given that $\tan (\alpha\pi/2) = 1/(\cot (\alpha\pi/2))$
is strictly increasing in $\alpha$
and
\begin{equation}
  \tan \frac{(1-\alpha^*)\pi}{2}
  =
  \cot \frac{\alpha^*\pi}{2}
  =
  s^*,
\end{equation}
this completes the proof.
\end{proof}

\subsection*{Conclusion}

Let $m \geq 1$ be an integer.
By \ref{k=3_decomp}, \ref{alpha=0},
and \ref{ind_null_rel_threshold}
\begin{equation}
\begin{aligned}
  \nul(\ks_{3,m})
  &=
  3 + 2\card[(m^{-1}\Z) \cap \{\alpha^*, ~1-\alpha^*\}]
  \\
  &=
  3
    +
    \begin{cases}
      4 &\mbox{if } m\alpha^* \in \Z
      \\
      0 &\mbox{if } m\alpha^* \not\in \Z,
    \end{cases}
  \\
  \ind(\ks_{3,m})
  &=
  3 + \sum_{n=1}^{m-1} \ind_{n/m}(\ks_3)
  \\
  &=
  3 + 4(m-1)
  \\
    &\quad
    +2 \card[(m^{-1}\Z) \cap (\alpha^*, ~ 1-\alpha^*)]
  \\
  &=
  4m - 1
    +2[
      (m-\lfloor m\alpha^* \rfloor)
      -\lfloor m\alpha^* \rfloor
      -1
    ]
  \\
  &=
  6m - 3 - 4 \lfloor m\alpha^* \rfloor,
\end{aligned}
\end{equation}
completing the proof of \ref{k=3}.

\end{document}